\documentclass[11pt,psamsfonts]{amsart}
\usepackage{amsmath}
\usepackage{amsthm}
\usepackage{amssymb}
\usepackage{amscd}
\usepackage{amsfonts}
\usepackage{amsbsy}
\usepackage{graphicx}
\usepackage[dvips]{psfrag}
\usepackage{array}
\usepackage{color}
\usepackage{epsfig}
\usepackage{url}

\newcommand{\R}{\ensuremath{\mathbb{R}}}

\newtheorem {theorem} {Theorem}
\newtheorem {proposition} [theorem]{Proposition}
\newtheorem {corollary} [theorem]{Corollary}
\newtheorem {lemma}  [theorem]{Lemma}

\title[Planar quasi--homogeneous differential systems]
{Global dynamics of planar quasi--homogeneous differential systems}

\author[Y. Tang, X. Zhang]
{Yilei Tang$^1$ and Xiang Zhang$^2$}

\address{$^1$ Department of Mathematics,  Shanghai Jiao Tong University, Shanghai, 200240, P. R. China}
\email{mathtyl@sjtu.edu.cn}

\address{$^2$ Department of Mathematics,  and MOE--LSC, Shanghai Jiao Tong University, Shanghai, 200240, P. R. China}
\email{xzhang@sjtu.edu.cn}



\begin{document}

\begin{abstract}
 In this paper we provide a new method to study global dynamics of planar quasi--homogeneous differential systems.
We first prove that all planar quasi--homogeneous polynomial differential systems can be translated into homogeneous differential systems
and show that  all quintic quasi--homogeneous but non--homogeneous systems can be reduced to four homogeneous ones.
Then we present some properties of homogeneous systems, which can be used to discuss the dynamics of quasi--homogeneous systems.
Finally we characterize the global topological phase portraits of quintic quasi--homogeneous but non--homogeneous differential systems.
\end{abstract}

\maketitle

\section{ Introduction }\label{s1}

In the qualitative theory of planar polynomial differential systems, there are lots of results on their global topological structures. But there are only few class of planar polynomial differential systems whose topological phase portraits were completely characterized. This paper will focus on the global structures of  quasi--homogeneous polynomial differential systems.

Consider a real planar polynomial differential system
\begin{eqnarray}  \label{quasi5-general}
 \dot{x}¢« = P(x, y), \qquad\qquad  \dot{y}¢« = Q(x, y),
\end{eqnarray}
where  $P(x, y), Q(x, y) \in \R [x, y]$, $PQ \not\equiv 0$ and the origin $O=(0, 0)$  is a singularity of system \eqref{quasi5-general}.
As usual, the dot denotes derivative with respect to an independent real variable $t$ and
$\mathbb{R}[x, y]$ denotes the ring of polynomials in
the variables $x$ and $y$ with coefficients in $\R$.
We say that  system (\ref{quasi5-general}) has {\it degree $n$} if
$n=\max\{\mbox{deg}\,P,\mbox{deg}\, Q\}$.  In what follows we assume without loss of generality
that $P$ and $Q$ in system \eqref{quasi5-general} have not a non--constant common factor.

System (\ref{quasi5-general}) is called a {\it quasi--homogeneous
polynomial differential system} if there exist constants
$s_1,s_2,d\in \mathbb N$ such that for an arbitrary
$\alpha\in
\R _+$ it holds that
\begin{equation}  \label{alphaPQ}
P(\alpha^{s_1}x, \alpha^{s_2}y) =
\alpha^{s_1+d-1} P(x, y), \ \  \ \  Q(\alpha^{s_1}x, \alpha^{s_2}y)=
\alpha^{s_2+d-1}  Q(x, y),
\end{equation}
where $\mathbb N$ is the set of positive integers and $\mathbb R_+$ is the set of positive real numbers. We call
$(s_1,s_2)$ {\it weight exponents} of system (\ref{quasi5-general})
and $d$ {\it weight degree} with respect to the weight exponents.
Moreover,  $w =(s_1, s_2, d)$ is denominated {\it weight vector} of system
(\ref{quasi5-general}) or of its associated vector field. For a
quasi--homogeneous polynomial differential system
(\ref{quasi5-general}), a weight vector $
\widetilde{w}=(\widetilde{s}_1 , \widetilde{s}_2, \tilde{d})$ is
{\it minimal} for system (\ref{quasi5-general}) if any other weight
vector $(s_1, s_2, d)$ of system (\ref{quasi5-general}) satisfies
$\tilde{s}_1\le s_1, \tilde{s}_2\le s_2$ and $\tilde{d}\le d$.
Clearly each quasi--homogeneous polynomial differential system has a
unique minimal weight vector. When $s_1=s_2=1$, system
(\ref{quasi5-general}) is a homogeneous one of degree $d$.

Quasi--homogeneous  polynomial differential systems have  been
intensively investigated by many different authors  from integrability point of
view, see for example \cite{Alga2009, GGL2013, Gori1996, Hu2007} and the references therein.
It is well known that all planar quasi--homogeneous vector fields are Liouvillian
integrable, see e.g. \cite{Garc2003,Garc2013,Li2009}. Specially, for
the polynomial and rational integrability of planar
quasi--homogeneous vector fields we refer readers to \cite{Alga2010,
Cair2007, Llib2002, Yosh1983}, and for the center and limit cycle
problems we refer to \cite{Alga2012, Gavr2009, Li2009} and the references therein.

Homogeneous differential systems are a class of special quasi--homogeneous polynomial differential systems of form \eqref{quasi5-general} with
$\mbox{deg} \, P=\mbox{deg} \,Q$, which  have also been studied by several authors, see e.g.
\cite{Date1979, New1978, Sib1977, Vdo1984, Ye1986} for quadratic homogeneous systems, \cite{Cima1990, Ye1995} for cubic homogeneous ones,
and \cite{Cima1990,  Llib1996} for homogeneous ones of arbitrary degree.
These papers have either characterized the phase portraits of homogeneous polynomial vector fields of degrees 2 and 3,
or obtained the algebraic classification of homogeneous vector fields or characterized the structurally stable homogeneous vector fields.

Recently, Garc\'ia {\it et al} \cite{Garc2013} provided an
algorithm to compute quasi--homogeneous but non--homogeneous polynomial differential systems with a given degree
and obtained all the quadratic and cubic quasi--homogeneous but non--homogeneous vector fields.
Aziz {\it et al} \cite{Aziz2014}  characterized all cubic
quasi--homogeneous polynomial differential equations which have a center. Liang {\it et al} \cite{Liang2014} classified all
quartic quasi--homogeneous but non--homogeneous differential systems, and obtained all their topological phase portraits.
Tang {\it et al} \cite{TWZ2015} presented all quintic quasi--homogeneous but non--homogeneous differential systems, and characterized their center problem.

Until now the topological phase portraits of all quintic quasi--homogeneous but non--homogeneous differential systems have not been settled.
As we checked, it is difficult to apply the methods in \cite{Aziz2014, Liang2014} to deal with this problem.
For doing so, we will provide a new method to study topological structure of  quasi--homogeneous but not homogeneous differential systems.
First we prove a general result and show that all quasi--homogeneous differential systems  can be translated into
homogeneous ones. Secondly we characterize the phase portraits of quintic quasi--homogeneous but non--homogeneous polynomial
vector fields.

This article is organized as follows. Section \ref{s2} will concentrate on  properties of quasi--homogeneous
 polynomial differential systems. There we provide expressions of all homogeneous differential systems which are transformed from quintic quasi--homogeneous differential systems.
 Section \ref{s3}  introduces the methods of  blow--up and normal sectors to research the global properties of homogeneous differential systems,
 where we obtain some results, which partially improve the results in \cite{Cima1990}.
The last section  is devoted to  study the  global structures and phase portraits of quintic quasi--homogeneous differential systems.



\bigskip


\section{Properties of  quasi--homogeneous vector fields} \label{s2}

The main aim of this section is to prove that all quasi--homogeneous but non--homogeneous differential systems can be translated into
homogeneous differential systems and to apply this result to study topological structure of quintic  quasi--homogeneous differential systems.
For doing so, we need to introduce the generic form of a quasi--homogeneous but non--homogeneous differential system of degree $n$.
Without loss of generality we assume that $s_1 > s_2$,
otherwise we can exchange the coordinates $x$ and $y$.

By \cite[Proposition 10]{Garc2013}, if system \eqref{quasi5-general} is quasi--homogeneous but non--homogeneous of degree $n$ with the weight vector
 $(s_1,s_2,d)$ and $d>1$, then the system has the minimal weight vector
\begin{eqnarray} \label{tw}
\widetilde{w} = \left(\frac{ \varsigma
+ \kappa}{s}, \,\,\, \frac{\kappa}{s},\,\,\, 1 + \frac{(p -1)\varsigma + (n - 1)\kappa}{s}\right),
\end{eqnarray}
with $p\in \{0,1,...,n-1\}$, $\varsigma \in \{1,2,...,n-p\}$ and $\kappa\in \{1,\ldots,n-p-\varsigma+1\}$ satisfying
 \[
 s_1=\frac{(\varsigma+\kappa)(d-1)}{D}, \qquad s_2=\frac{\kappa(d-1)}{D},
 \]
 where $D= (p -1)\varsigma + (n - 1)\kappa$  and $s=\gcd(\varsigma,\kappa)$. Furthermore, by the algorithm posed in subsection 3.1 of \cite{Garc2013}
 the quasi--homogeneous but non--homogeneous differential system \eqref{quasi5-general} of degree $n$
 with the weight vector $(s_1,s_2,d)$ can be written as
\begin{eqnarray} \label{quasi-sys}
X_{p\varsigma \kappa}=X_n^p+X_{n-\varsigma}^{p\varsigma \kappa} +V_{p\varsigma \kappa},
\end{eqnarray}
where
\[
X_n^p=(a_{p,n-p}x^py^{n-p}, b_{p-1,n-p+1}x^{p-1}y^{n-p+1})^T
\]
is the homogeneous part of degree $n$ with coefficients not simultaneous vanishing,
\[
V_{p\varsigma \kappa}=
\sum\limits_{\scriptsize\begin{array}{c} \varsigma_1\in\{1,\ldots,n-p\}\setminus\{\varsigma\}
\\ \kappa_{\varsigma_1}\varsigma=\kappa \varsigma_1 \mbox{ and }
\\ \kappa_{\varsigma_1}\in\{1,\ldots,n-\varsigma_1-p+1\}\end{array}} X_{n-\varsigma_1}^{p\varsigma_1\kappa_{\varsigma_1}},
\]
\[
X_{n-\varsigma}^{p\varsigma \kappa}=(a_{p+\kappa,n-\varsigma-p-\kappa}x^{p+\kappa}y^{n-\varsigma-p-\kappa}, \,
b_{p+\kappa-1,n-\varsigma-p-\kappa+1}x^{p+\kappa-1}y^{n-\varsigma-p-\kappa+1})^T,
\]
and
$X_{n-\varsigma_1}^{p\varsigma_1\kappa_{\varsigma_1}}$'s having the same expressions as that of $X_{n-\varsigma}^{p\varsigma \kappa}$.
In order for $X_{p\varsigma\kappa}$ to be quasi--homogeneous but non--homogeneous of degree $n$ we must have $X_n^p\not\equiv 0$
and at least one of the other elements not identically vanishing.

Using the above algorithm and the associated notations, we can prove the next result.

\medskip

\begin{theorem}\label{th-homo}
Any quasi--homogeneous but non--homogeneous polynomial differential system \eqref{quasi5-general} of degree $n$
can be transformed into a homogeneous polynomial differential system  by a change of variables being of the composition
of the transformation $ \tilde{x}=(\pm x)^{\frac{s_2}{\beta}},\quad  ~~ \tilde{y}=(\pm y)^{\frac{s_1}{\beta}}$, where $\beta$ is a suitable nonnegative integer.
\end{theorem}

\begin{proof}
Assume that the quasi--homogeneous system of degree $n$ has the minimal  weight vector
 $(s_1,s_2,d)$.  We will distinguish two cases: $d>1$ and $d=1$.

\noindent {\it Case $1$}. $d>1$. According to discussion before Theorem \ref{th-homo}, we only need to prove that the following system
\begin{equation}\label{sys1}
\left(\begin{array}{l}
\dot{x}\\
\dot{y}
\end{array}\right)
=\left(\begin{array}{l}
a_{p,n-p}x^py^{n-p} +a_{p+\kappa,n-\varsigma-p-\kappa}x^{p+\kappa}y^{n-\varsigma-p-\kappa}
\\
 \qquad  +a_{p+\kappa_1,n-\varsigma_1-p-\kappa_1}x^{p+\kappa_1}y^{n-\varsigma_1-p-\kappa_1}
 \\
 b_{p-1,n-p+1}x^{p-1}y^{n-p+1} +b_{p+\kappa-1,n-\varsigma-p-\kappa+1}x^{p+\kappa-1}y^{n-\varsigma-p-\kappa+1}
\\
\qquad  +b_{p+\kappa_1-1,n-\varsigma_1-p-\kappa_1+1}x^{p+\kappa_1-1}y^{n-\varsigma_1-p-\kappa_1+1}
\end{array}\right)
+  \widetilde{X},
\end{equation}
can be changed to a homogeneous differential system, where
\begin{eqnarray*}
  \widetilde{X} &:=&V_{p\varsigma \kappa}-Z_1,
\end{eqnarray*}
with
\[
 Z_1:=\left(\begin{array}{c}
 a_{p+\kappa_1,n-\varsigma_1-p-\kappa_1}x^{p+\kappa_1}y^{n-\varsigma_1-p-\kappa_1}\\
b_{p+\kappa_1-1,n-\varsigma_1-p-\kappa_1+1}x^{p+\kappa_1-1}y^{n-\varsigma_1-p-\kappa_1+1}
\end{array}\right)
\]
an arbitrary chosen vector valued homogeneous part in $V_{p\varsigma \kappa}$ (see  \eqref{quasi-sys}).
Note that
\begin{align*}
&b_{p-1,n-p+1}=0,\qquad &\mbox{ if } &  p=0,
\\
&a_{p+\kappa_1,n-\varsigma_1-p-\kappa_1}=0,\qquad & \mbox{ if } & p=n+1-\varsigma_1-\kappa_1 .
\end{align*}

For $(x, y)\in \R _+^2$, taking the change of variables
\begin{eqnarray} \label{change1}
 \tilde{x}=x^{\frac{s_2}{\beta}},\quad  ~~ \tilde{y}=y^{\frac{s_1}{\beta}},
\end{eqnarray}
where $\beta \in \mathbb N$ will be determined later on,
we get from system  \eqref{sys1} that
\begin{equation}\label{sys2}
\left(\begin{array}{c}
\dot{x}\\
\dot{y}
\end{array}\right)
=\left(\begin{array}{l}
\frac{s_2}{\beta} \left( a_{p,n-p} x^{\frac{\beta(p-1)}{s_2}+1} y^{\frac{\beta(n-p)}{s_1}}
+a_{p+\kappa,n-\varsigma-p-\kappa}x^{\frac{\beta(p+\kappa-1)}{s_2}+1}y^{\frac{\beta(n-\varsigma-p-\kappa)}{s_1}}\right.
\\
\left.\qquad\quad +a_{p+\kappa_1,n-\varsigma_1-p-\kappa_1}x^{\frac{\beta(p+\kappa_1-1)}{s_2}+1}y^{\frac{\beta(n-\varsigma_1-p-\kappa_1)}{s_1}} \right)
 \\
\frac{s_1}{\beta} \left(b_{p-1,n-p+1}x^{\frac{\beta(p-1)}{s_2}}y^{\frac{\beta(n-p)}{s_1}+1}\right.
\\
\quad\qquad +b_{p+\kappa-1,n-\varsigma-p-\kappa+1}x^{\frac{\beta(p+\kappa-1)}{s_2}}y^{\frac{\beta(n-\varsigma-p-\kappa)}{s_1}+1}
\\
\left.\quad\qquad +b_{p+\kappa_1-1,n-\varsigma_1-p-\kappa_1+1}x^{\frac{\beta(p+\kappa_1-1)}{s_2}}y^{\frac{\beta(n-\varsigma_1-p-\kappa_1)}{s_1}+1} \right)
\end{array}\right)
+  \widetilde{X},
\end{equation}
where for simplicity we still use $x$ and $y$ replacing $\tilde{x}$ and $\tilde{y}$. We remark that here taking $(x, y) \in \R _+^2$ is only for simplifying notations. In the other cases, for example $x<0$ and $y<0$ we can take  $\tilde{x}=-(-x)^{\frac{s_2}{\beta}},\quad  ~~ \tilde{y}=-(-y)^{\frac{s_1}{\beta}}$.

Next, by \eqref{tw} and \eqref{quasi-sys} we could see
\[
s_1= \frac{ \varsigma + \kappa}{s},\quad s_2=\frac{\kappa}{s}, \quad \mbox{ and }\quad  \kappa_1\varsigma=\kappa\varsigma_1.
\]
These imply that
\begin{equation}\label{equl12}
s_1\kappa=s_2(\varsigma+\kappa) \ \ ~~\mbox{and}~~ \  \  \frac{\kappa}{s_2}-\frac{\varsigma+\kappa}{s_1}=\frac{\kappa_1}{s_2}-\frac{\varsigma_1+\kappa_1}{s_1}.
\end{equation}
Hence the degrees of the terms in the equation of $\dot{x}$ in \eqref{sys2} hold:
\begin{align}\label{equl1}
\mbox{deg} \left( x^{\frac{\beta(p-1)}{s_2}+1} y^{\frac{\beta(n-p)}{s_1}}\right)
& =
\mbox{deg} \left(x^{\frac{\beta(p+\kappa-1)}{s_2}+1}y^{\frac{\beta(n-\varsigma-p-\kappa)}{s_1}}\right)\nonumber
\\&=
\mbox{deg}  \left(x^{\frac{\beta(p+\kappa_1-1)}{s_2}+1}y^{\frac{\beta(n-\varsigma_1-p-\kappa_1)}{s_1}}\right).
\end{align}
Thus, all terms on the right hand side of the equation of $\dot{x}$ in \eqref{sys2} have the  same degree, because $ Z_1$
is an arbitrarily chosen homogeneous vector field in $V_{p\varsigma \kappa}$.

A similar calculation as above shows that
\begin{align}\label{equl2}
\mbox{deg} \left(x^{\frac{\beta(p-1)}{s_2}}y^{\frac{\beta(n-p)}{s_1}+1}\right) &=
\mbox{deg} \left(x^{\frac{\beta(p+\kappa-1)}{s_2}}y^{\frac{\beta(n-\varsigma-p-\kappa)}{s_1}+1}\right)\nonumber
\\
&=
\mbox{deg} \left(x^{\frac{\beta(p+\kappa_1-1)}{s_2}}y^{\frac{\beta(n-\varsigma_1-p-\kappa_1)}{s_1}+1} \right)
\end{align}
in the equation of $\dot{y}$ in \eqref{sys2},  which are also equivalent to
\eqref{equl12}.
Furthermore, notice that the degrees of the second terms on the right hand of $\dot{x}$ and $\dot{y}$ in \eqref{sys2} respectively are same, i.e.,
\begin{equation}\label{equl3}
\mbox{deg} ~  \left(x^{\frac{\beta(p+\kappa-1)}{s_2}+1}y^{\frac{\beta(n-\varsigma-p-\kappa)}{s_1}}\right)
=\mbox{deg}~ \left( x^{\frac{\beta(p+\kappa-1)}{s_2}}y^{\frac{\beta(n-\varsigma-p-\kappa)}{s_1}+1}\right).
\end{equation}
Hence, it follows from \eqref{equl1}, \eqref{equl2} and \eqref{equl3} that all terms in system  \eqref{sys2} have the same degree.

 Now,  we only need to prove that system \eqref{sys2} is a polynomial differential
system  after choosing an appropriate integer $\beta$.

When $0< p<n+1-\varsigma_1-\kappa_1$,
by the expression of the degree mentioned above we can  choose $\beta$ to be the least common multiple
of $s_1$ and $s_2$, i.e.,
\begin{equation}\label{beta}
\beta=\mbox{lcm} (s_1, s_2),
\end{equation}
because $n, s_1, s_2 , \kappa, \kappa_1, \varsigma, \varsigma_1\in \mathbb N$.

When $0=p<n+1-\varsigma_1-\kappa_1$ (resp. $0<p=n+1-\varsigma_1-\kappa_1$, $0=p=n+1-\varsigma_1-\kappa_1$), we still choose $\beta=\mbox{lcm} (s_1, s_2)$.
After the time rescaling $dt=x^{\frac{\beta}{s_2}-1} dt_1$ (resp. $dt=y^{\frac{\beta}{s_1}-1} dt_1$, $dt=x^{\frac{\beta}{s_2}-1} y^{\frac{\beta}{s_1}-1} dt_1$),
system \eqref{sys2} is translated to a polynomial differential system.

\medskip

\noindent{\it Case $2$}. $d=1$. From \cite[Proposition 9]{Garc2013}, the minimal  weight vector of the quasi--homogeneous but non--homogeneous differential
 system of degree $n$ is  $(n, 1, 1)$ and the  system  is
 \begin{eqnarray}
 \label{sys-d1}
 \dot{x}=a_{0n}y^n+a_{10}x, ~~ \ \ \dot{y}=b_{01}y,
\end{eqnarray}
where the coefficients $a_{0n}$, $a_{10}$ and $b_{01}$ are not all equal to zero.
Applying  the transformation
\begin{eqnarray} \label{change-d1}
 \tilde{x}=x^{\frac{1}{n}}, \   \  ~~ \tilde{y}=y,
\end{eqnarray}
together with the time scaling $dt=\tilde{x}^{n-1}dt_1$, system \eqref{sys-d1} becomes
 \begin{eqnarray}
 \label{homo-d1}
 \dot{x}=\frac{1}{n}(a_{0n}y^n+a_{10}x^n), \ \  \ ~~ \dot{y}=b_{01}yx^{n-1},
\end{eqnarray}
where we still use $x, y$ replacing $\tilde{x}, \tilde{y}$ for simplicity.

We complete the proof of the theorem. \end{proof}
\smallskip

We note that the changes of variables \eqref{change1} and \eqref{change-d1} may not be smooth at $x=0$ and $y=0$. But we can apply them for studying topological structure of the systems as shown below.

\medskip


By Theorem \ref{th-homo}, any quasi--homogeneous but  non--homogeneous
polynomial differential system \eqref{quasi-sys} of degree $n$ can be transformed into
a homogeneous differential system.
Actually, we can ascertain the degree of the homogeneous system obtained above.

\medskip

\begin{theorem}\label{prop1}
An arbitrary quasi--homogeneous but non--homogeneous
polynomial differential system  \eqref{quasi5-general}   of degree $n$ with the minimal weight vector  $ \widetilde{w}=(s_1, s_2, d)$ can be transformed to
a homogeneous differential system of degree either $d$, or $d+s_1-1$ or $d+s_2-1$ or $d+s_1+s_2-2$ by the change \eqref{change1}.
\end{theorem}

\begin{proof}
First of all, we claim that $s_1$ and $s_2$ are coprime.
Indeed, by contrary, if $s_1$ and $s_2$ have a great common divisor larger than $1$, set
\[
r=\mbox{gcd} (s_1, s_2)>1.
\]
Since $P\not\equiv 0$, without loss of generality, let
$a_{p+\tilde{\kappa},n-\tilde{\varsigma}-p-\tilde{\kappa}}x^{p+\tilde{\kappa}}y^{n-\tilde{\varsigma}-p-\tilde{\kappa}}$
be a nonvanishing monomial of $P(x,y)$ in system  \eqref{quasi5-general}, where $\tilde{\varsigma}, p, \tilde{\kappa} \in \mathbb N \cup \{0\}$
and $n\in\mathbb N$. Thus,
\begin{equation}\label{ssd}
(p+\tilde{\kappa}-1)s_1 +(n-\tilde{\varsigma}-p-\tilde{\kappa})s_2=d-1,
\end{equation}
because system (\ref{quasi5-general}) is  quasi--homogeneous and $P$ satisfies \eqref{alphaPQ}.
It indicates that $r$ also divides $d-1$.
So, we have
\begin{equation}\label{ssd2}
s_1=r\tilde{s}_1, ~~s_2=r\tilde{s}_2, ~~d-1=r(\tilde{d}-1),
\end{equation}
where $\tilde{s}_1, \tilde{s}_2, \tilde{d} \in \mathbb N $ and $\tilde{s}_1<s_1, \tilde{s}_2<s_2, \tilde{d}<d$.
Using equalities \eqref{alphaPQ} again together with \eqref{ssd2}, we obtain
\begin{align} \label{ssdP}
 P(\alpha^{s_1}x, \alpha^{s_2}y) &= P(\alpha^{r\tilde{s}_1}x, \alpha^{r\tilde{s}_2}y)
 =\alpha^{s_1+(d-1)} P(x, y)\nonumber
 \\
&=  \alpha^{r\tilde{s}_1+r(\tilde{d}-1)} P(x, y) =(\alpha^r)^{ \tilde{s}_1+(\tilde{d}-1)} P(x, y).
\end{align}
A similar calculation also shows that
\begin{eqnarray} \label{ssdQ}
\begin{split}
 Q((\alpha^r)^{\tilde{s}_1}x, (\alpha^r)^{\tilde{s}_2}y)  =(\alpha^r)^{ \tilde{s}_2+(\tilde{d}-1)} Q(x, y).
\end{split}
\end{eqnarray}
Hence, by \eqref{ssdP} and  \eqref{ssdQ}, we get that the vector $(\tilde{s}_1, \tilde{s}_2, \tilde{d})$ is also a weight vector  of quasi--homogeneous
 system  \eqref{quasi5-general} and $\tilde{s}_1<s_1, \tilde{s}_2<s_2, \tilde{d}<d$, a contradiction with the fact that
$ \widetilde{w}=(s_1, s_2, d)$ is  the minimal weight vector  of system  \eqref{quasi5-general}. Therefore,
 $s_1$ and $s_2$ have not a common divisor. The claim follows.

The above shows that  $\mbox{lcm} (s_1, s_2)=s_1s_2$. From the choice of $\beta$ as in \eqref{beta}, we get $\beta=s_1s_2$.
In the following, we discuss the degree of system \eqref{sys2} in two cases: $d>1$ and $d=1$.

\noindent{\it Case $1$}. $d>1$.
Then, it  follows from \eqref{equl1} and  \eqref{ssd} that
\[
\mbox{deg} ~\left( x^{\frac{\beta(p+\tilde{k}-1)}{s_2}+1} y^{\frac{\beta(n-\tilde{\varsigma}-p-\tilde{k})}{s_1}}\right)
=(p+\tilde{k}-1)s_1+1+(n-\tilde{\varsigma}-p-\tilde{k})s_2= d,
\]
Hence, we can translate a general quasi--homogeneous but non--homogeneous
differential system  \eqref{quasi5-general} of degree $n$  into
a homogeneous differential system of degree $d$ when $0< p<n+1- \varsigma_1-\kappa_1$.

If $0=p<n+1-\varsigma_1-\kappa_1$,
some calculations show that
\[
\mbox{deg} ~\left( x^{\frac{-\beta}{s_2}+1} y^{\frac{\beta n}{s_1}}\right)+s_1-1 =-s_1+1+ns_2+s_1-1= d+s_1-1,
\]
by substituting $p=\tilde{\varsigma}=\tilde{\kappa}=0 $ in \eqref{ssd} and using \eqref{equl1} again  together with
 a time rescaling  $dt=x^{\frac{\beta}{s_2}-1} dt_1=x^{s_1-1} dt_1$.
At this time,  the  quasi--homogeneous but  non--homogeneous
differential system  \eqref{quasi5-general}   of degree $n$ can be transformed to
a homogeneous differential system of degree $d+s_1-1$.

When  $0<p=n+1-\varsigma_1-\kappa_1$ (resp.  $0=p=n+1-\varsigma_1-\kappa_1$), the quasi--homogeneous but  non--homogeneous
differential system  \eqref{quasi5-general}   of degree $n$ can be transformed to
a homogeneous system of degree $d+s_2-1$ (resp. $d+s_1+s_2-2$) by a similar calculation as that in the case $0=p<n+1-\varsigma_1-\kappa_1$.

\noindent{\it Case $2$}. $d=1$, the  quasi--homogeneous but  non--homogeneous differential
 system of degree $n$ can be transformed to a homogeneous differential system of degree $d+s_1-1$, see \eqref{homo-d1}.

We complete the proof of the theorem.
\end{proof}

\medskip


The following result is an easy consequence of Theorem \ref{prop1}.

\medskip

\begin{corollary}\label{coro1}
If a  quasi--homogeneous but  non--homogeneous
differential system  \eqref{quasi5-general}   of degree $n$ has the minimal weight vector $ \widetilde{w}=(s_1, s_2, d)$,
then at least one of $s_1$ and $s_2$ is odd.
\end{corollary}

\medskip

By Corollary \ref{coro1} at most one of $s_1$ and $s_2$ is even. If one of them is even, we assume without loss of generality that
$s_2$ is even.  Then the transformation \eqref{change1} can be written as
\begin{eqnarray*} 
 \tilde{x}=x^{\frac{1}{s_1}}, ~~
 \tilde{y}= \left\{
\begin{array}{lll}
& y^{\frac{1}{s_2}} &  ~\mbox{if}~ y>0,
 \\
& -(-y)^{\frac{1}{s_2}}  & ~\mbox{if}~ y<0.
\end{array}
\right.
\end{eqnarray*}
Of course, if $s_2$ is odd, then $\tilde y=y^{\frac 1{s_2}}$.
Clearly the change \eqref{change1} is a one to one correspondence in the half plane $y>0$ provided that $s_1$ is odd.
Then, for $(x, y)\in \R ^2$, we can change any quasi--homogeneous but non--homogeneous polynomial differential system \eqref{quasi5-general} of degree $n$
 into a homogeneous polynomial differential system in Theorem  \ref{th-homo}.

\medskip

 The following theorem guarantees that quasi--homogeneous systems
are symmetric with respect to either an axis or the origin with possibly a time reverse. Using this fact we can obtain global dynamics of a quasi--homogeneous differential system through the dynamics of its associated homogeneous systems in the half plane.

\medskip

\begin{theorem}
\label{th-invariant}
An arbitrary quasi--homogeneous but non--homogeneous
polynomial differential system of degree $n$ with the minimal  weight vector
 $(s_1, s_2, d)$ is invariant either under the change $(x, y)\to (x, -y)$ if $s_1$ is even and $s_2$ is odd,
or under the change $(x, y)\to (-x, y)$ if $s_1$ is odd and $s_2$ is even,
 or under the change $(x, y)\to (-x, -y)$ if both $s_1$ and $s_2$ are odd,  without taking into account the direction of the time.
\end{theorem}

\begin{proof}
First we discuss the case $d>1$, $s_1$ is even and $s_2$ is odd  for the quasi--homogeneous system \eqref{sys1}.
Taking the change of variables $(\tilde{x}, \tilde{y})= (x, -y)$, we  rewrite system \eqref{sys1} as
\begin{equation}\label{sys11}
\left(\begin{array}{c}
\dot{x}\\
\dot{y}
\end{array}\right)=
\left(\begin{array}{l}
 (-1)^{n-p}a_{p,n-p}x^py^{n-p} \\
\quad +(-1)^{n-\varsigma-p-\kappa}a_{p+\kappa,n-\varsigma-p-\kappa}x^{p+\kappa}y^{n-\varsigma-p-\kappa}
\\
\quad  +(-1)^{n-\varsigma_1-p-\kappa_1}a_{p+\kappa_1,n-\varsigma_1-p-\kappa_1}x^{p+\kappa_1}y^{n-\varsigma_1-p-\kappa_1}
 \\
 (-1)^{n-p}b_{p-1,n-p+1}x^{p-1}y^{n-p+1}
\\
\quad +(-1)^{n-\varsigma-p-\kappa}b_{p+\kappa-1,n-\varsigma-p-\kappa+1}x^{p+\kappa-1}y^{n-\varsigma-p-\kappa+1}
\\
\quad  +(-1)^{n-\varsigma_1-p-\kappa_1}b_{p+\kappa_1-1,n-\varsigma_1-p-\kappa_1+1}x^{p+\kappa_1-1}y^{n-\varsigma_1-p-\kappa_1+1}
\end{array}\right)
+  \widetilde{X},
\end{equation}
where we still write $\tilde{x}$, $\tilde{y}$ as $x$, $y$ respectively for simplicity to notations.

Now, we only need to prove
 $(-1)^{\varsigma+\kappa}=(-1)^{\varsigma_1+\kappa_1}=1$, i.e., both $\varsigma+\kappa$ and $\varsigma_1+\kappa_1$ are even.
In fact, from calculation of the minimal weight vector  \eqref{tw}  we obtain that
\begin{eqnarray*}
 \varsigma+\kappa=ss_1,\quad  ~~~ \varsigma_1+\kappa_1=\tilde{s}s_1,
\end{eqnarray*}
where $s=\gcd(\varsigma,\kappa)$ and $\tilde{s}=\gcd(\varsigma_1,\kappa_1)$.
These force that both $\varsigma+\kappa$ and $\varsigma_1+\kappa_1$ are even. Besides, the fact that
$X_{n-\varsigma_1}^{p\varsigma_1\kappa_{\varsigma_1}}$ is an arbitrarily chosen homogeneous part of $\tilde{X}$ indicates that
 system \eqref{sys11} is the same as  system \eqref{sys1} after a time rescaling $dt=d\tilde{t}/(-1)^{n-p}$.
This proves that any quasi--homogeneous but non--homogeneous
polynomial differential system is invariant under the change $(x, y)\to (x, -y)$ if $d>1$, $s_1$ is even and $s_2$ is odd.

For the remaining three cases: $d>1$, $s_1$ is odd and $s_2$ is even;   $d>1$, $s_1$ and $s_2$ are odd, and $d=1$, the proof is similar
to the case that $d>1$, $s_1$ is even and $s_2$ is odd. The details are omitted.
\end{proof}

\medskip



We remark that the  least common multiple of $s_1$ and $s_2$ is not  the unique choice for $\beta$,
although Theorem \ref{th-homo} and Theorem  \ref{prop1} give a method showing how to translate a quasi--homogeneous but non--homogeneous
differential system into a homogeneous one. Actually,  we can select any  common multiple of $s_1$ and $s_2$ for $\beta$. Besides, if
the quasi--homogeneous but non--homogeneous differential system with the minimal weight vector $\widetilde{w}=(s_1, s_2, d)$
has spare terms, we could choose $\beta=1$ in \eqref{change1}, which together with a time rescaling
 translates the system to a homogeneous one with degree less than that of the homogeneous system
 translated by using the change \eqref{change1}  with $\beta=\mbox{lcm} (s_1, s_2)$.
  Next examples demonstrate that for some systems we can use the change of variables
\begin{eqnarray} \label{change2}
 \tilde{x}=x^{s_2}, \qquad  ~~ \tilde{y}=y^{s_1}.
\end{eqnarray}

Consider a  quasi--homogeneous  system
\begin{equation}\label{e151018-1}
 \dot{x}=a_{05}y^5+a_{13}xy^3+a_{21}x^2y, ~~
\dot{y}=b_{04}y^4+b_{12}xy^2+b_{20}x^2,
\end{equation}
in $\R _+^2$, with $a_{05}b_{20} \ne0$ and the minimal weight vector $ \widetilde{w}=(2, 1, 4)$.
Taking the change \eqref{change1} with $\beta=\mbox{lcm} (s_1, s_2)$ and the time rescaling $dt=\tilde{x }dt_1$, system \eqref{e151018-1}
is translated to the homogeneous one of degree $5$
\begin{eqnarray*}
 \dot{x}=\left( a_{05}y^5+a_{13}x^2y^3+a_{21}x^4y \right)/2, ~~ \
\dot{y}=b_{04}xy^4+b_{12}x^3y^2+b_{20}x^5.
\end{eqnarray*}
However, system \eqref{e151018-1} is translated to the homogeneous one of degree $2$
\begin{eqnarray*}
 \dot{x}= a_{05}y^2+a_{13}xy+a_{21}x^2, ~~
\quad \dot{y}=2(b_{04}y^2+b_{12}xy+b_{20}x^2),
\end{eqnarray*}
by the change \eqref{change2} together with the time rescaling $dt=dt_1 /\tilde{y}^{\frac{1}{2}} $. Hereafter, we still write $\tilde{x}, \tilde{y}$ as $x, y$ for the change of variables $(x,y)\rightarrow (\widetilde x,\widetilde y)$.

Consider another  quasi--homogeneous  system
\begin{eqnarray*}
\dot{x}=a_{14}xy^4, ~~ \qquad
\dot{y}=b_{05}y^5+b_{40}x^4
\end{eqnarray*}
in $\R _+^2$, with $a_{14}b_{05}b_{40} \ne0$ and the minimal weight vector $ \widetilde{w}=(5,4,17)$. It
is translated to the homogeneous system of degree $20$
\begin{eqnarray*}
 \dot{x}= a_{14} xy^{19}/5, ~~ \qquad
\dot{y}=( b_{05}y^{20}+b_{40}x^{20} )/4,
\end{eqnarray*}
by the change \eqref{change1} with $\beta=\mbox{lcm} (s_1, s_2)$ and  the time rescaling $dt=\tilde{y}^3 dt_1$, and is translated to a linear homogeneous  system
\begin{eqnarray*}
 \dot{x}= 4a_{14}x, ~~ \qquad
\dot{y}=5(b_{05}y+b_{40}x)
\end{eqnarray*}
by the change \eqref{change2} together with the time rescaling $dt=dt_1 /\tilde{y}^{\frac{4}{5}} $.

These examples show that the change \eqref{change1} can always transform a quasi--homogeneous but non--homogeneous differential systems into a homogeneous one,
but $\beta=\mbox{lcm} (s_1, s_2)$ in the change \eqref{change1} may not be a best choice.

The next result, due to Tang et al \cite{TWZ2015}, characterizes all quasi--homogeneous but non--homogeneous quintic
polynomial differential systems.

\begin{lemma}\label{lm-q5list}
Every planar quintic quasi--homogeneous but non--homogeneous
polynomial differential system is one of the following $15$ systems:
{\footnotesize
\begin{eqnarray*}
& X_{011}:  &\dot{x}=a_{05}y^5+a_{13}xy^3+a_{21}x^2y, ~~
\dot{y}=b_{04}y^4+b_{12}xy^2+b_{20}x^2,\\
&& ~{\it with}~ a_{05}b_{20} \ne0 ~ \mbox{\it and the weight vector}~ \widetilde{w}=(2,1,4),
\\
& X_{012}:  &\dot{x}=a_{05}y^5+a_{22}x^2y^2, ~~
\dot{y}=b_{13}xy^3 +b_{30}x^3, \\
&& ~{\it with}~ a_{05}b_{30} \ne0 ~ \mbox{\it and the weight vector}~ \widetilde{w}=(3,2,8),
 \\
 & X_{014}:  &\dot{x}=a_{05}y^5+a_{40}x^4, ~~
\dot{y}=b_{31}x^3y, \\
&& ~{\it with}~ a_{05}a_{40}b_{31} \ne0 ~ \mbox{\it and the weight vector}~
\widetilde{w}=(5,4,16),
 \\
 & X_{015}:  &\dot{x}=a_{05}y^5, ~~
\dot{y}=b_{40}x^4,  \\
&& ~{\it with}~ a_{05}b_{40} \ne0 ~ \mbox{\it and the weight vector}~
\widetilde{w}=(6,5,20),
 \\
& X_{021}:  &\dot{x}=a_{05}y^5+a_{12}xy^2, ~~
\dot{y}=b_{03}y^3+b_{10}x,\\
&& ~{\it with}~ a_{05}b_{10} \ne0 ~ \mbox{\it and the weight vector}~ \widetilde{w}=(3,1,3),
\\
& X_{023}:  &\dot{x}=a_{05}y^5+a_{30}x^3, ~~
\dot{y}=b_{21}x^2y,\\
&& ~{\it with}~ a_{05}a_{30}b_{21} \ne0 ~ \mbox{\it and the weight vector}~
\widetilde{w}=(5,3,11),
\\
& X_{032}:  &\dot{x}=a_{05}y^5+a_{20}x^2, ~~
\dot{y}=b_{11}xy,\\
&& ~{\it with}~ a_{05}a_{20}b_{11} \ne0 ~ \mbox{\it and the weight vector}~
\widetilde{w}=(5,2,6),
\\
& X_{111}:  &\dot{x}=a_{14}xy^4 +a_{22}x^2y^2 +a_{30}x^3, ~~
\dot{y}=b_{05}y^5+b_{13}xy^3+b_{21}x^2y,\\
&& ~{\it with}~ a_{30}b_{05} \ne0 ~ \mbox{\it and the weight vector}~ \widetilde{w}=(2,1,5),
 \\
 & X_{113}:  &\dot{x}=a_{14}xy^4 +a_{40}x^4, ~~
\dot{y}=b_{05}y^5+b_{31}x^3y,\\
&& ~{\it with}~ a_{40}b_{05} \ne0 ~ \mbox{\it and the weight vector}~ \widetilde{w}=(4,3,13),
 \\
 & X_{114}:  &\dot{x}=a_{14}xy^4, ~~
\dot{y}=b_{05}y^5+b_{40}x^4,\\
&& ~{\it with}~ a_{14}b_{05}b_{40} \ne0 ~ \mbox{\it and the weight vector}~
\widetilde{w}=(5,4,17),
 \\
 & X_{123}:  &\dot{x}=a_{14}xy^4, ~~
\dot{y}=b_{05}y^5+b_{30}x^3,\\
&& ~{\it with}~ a_{14} b_{05}b_{30} \ne0 ~ \mbox{\it and the weight vector}~
\widetilde{w}=(5,3,13),
\\
\end{eqnarray*}
\begin{eqnarray*}
& X_{131}:  &\dot{x}=a_{14}xy^4 +a_{20}x^2, ~~
\dot{y}=b_{05}y^5+b_{11}xy,
\\
&& ~{\it with}~ a_{20}b_{05} \ne0 ~ \mbox{\it and the weight vector}~ \widetilde{w}=(4,1,5),
\\
& X_{132}: &\dot{x}=a_{14}xy^4, ~~
\dot{y}=b_{05}y^5+b_{20}x^2,\\
&& ~{\it with}~ a_{14}b_{05}b_{20} \ne0 ~ \mbox{\it and the weight vector}~
\widetilde{w}=(5,2,9),
\\
& X_{141}: &\dot{x}=a_{14}xy^4, ~~
\dot{y}=b_{05}y^5+b_{10}x,\\
&& ~{\it with}~ a_{14}b_{05}b_{10} \ne0 ~ \mbox{\it and the weight vector}~
\widetilde{w}=(5,1,5),
\\
& X_1:  &\dot{x}=a_{05}y^5+a_{10}x, ~~ \dot{y}=b_{01}y, \\
&&~{\it with}~
a_{05}a_{10}b_{01} \ne 0,  ~ \mbox{\it and the weight vector}~ \widetilde{w}=(5,1,1).
\end{eqnarray*}
}
\end{lemma}

Next result presents the homogeneous differential systems which are transformed from all quintic quasi--homogeneous but non--homogeneous
polynomial differential systems in Lemma \ref{lm-q5list}.


\begin{theorem}\label{th-homo-1}
Any quintic quasi--homogeneous but non--homogeneous polynomial differential system when restricted to either $y>0$ or $x>0$
can be transformed to one of the following four homogeneous systems
{\footnotesize
\begin{eqnarray*}
& \mathcal{H}_3:  &\dot{x}=x(c_{12}y^2 +c_{21}xy +c_{30}x^2), ~~
\dot{y}=y(d_{03}y^2+d_{12}xy+d_{21}x^2),
\\
&& ~{\it with}~ c_{30}d_{03} \ne0,
\\
& \mathcal{H}_2:  & \dot{x}= c_{02}y^2+c_{11}xy+c_{20}x^2, ~~
\dot{y}=d_{02}y^2+d_{11}xy+d_{20}x^2,\\
&& ~{\it with}~ c_{02}d_{20} \ne0, ~{\it or}~c_{02}=d_{20}=0, c_{20}d_{02} \ne0,
\\
& \mathcal{H}_1:  &\dot{x}=c_{01}y+c_{10}x, ~~
\dot{y}=d_{01}y +d_{10}x, \\
&& ~{\it with}~ c_{01}d_{10} \ne0,  ~{\it or}~d_{10}=0, c_{01}c_{10}d_{01} \ne0, ~{\it or}~c_{01}=0, c_{10}d_{01}d_{10}\ne 0,
 \\
 & \mathcal{H}_0:  &\dot{x}=c_0, ~~
\dot{y}=d_0,    ~{\it with}~ c_0d_0 \ne0.
\end{eqnarray*}
}
\end{theorem}

\begin{proof}
Using the change \eqref{change2} and the programm in the proof of Theorem \ref{th-homo}, some calculations show that the quintic quasi--homogeneous system $X_{ijk}$ in Lemma \ref{lm-q5list} can be correspondingly transformed into a homogeneous one, denoted by $\widetilde{X}_{ijk}$ or $\widetilde{X}_1$.

\medskip

{\footnotesize
\begin{eqnarray*}
& \widetilde{X}_{011}:  & \dot{x}= a_{05}y^2+a_{13}xy+a_{21}x^2, ~~
\dot{y}=2(b_{04}y^2+b_{12}xy+b_{20}x^2),\\
&& ~{\it with}~ a_{05}b_{20} \ne0, ~ {\it by}~ dt=dt_1 /y^{\frac{1}{2}}, ~y>0,
\\
& \widetilde{X}_{012}:  &\dot{x}=2(a_{05}y+a_{22}x), ~~
\dot{y}=3(b_{13}y +b_{30}x), \\
&& ~{\it with}~ a_{05}b_{30} \ne0, ~ {\it by}~ dt= dt_1/(x^{\frac{1}{2}}y^{\frac{2}{3}}), ~x>0,
 \\
 & \widetilde{X}_{014}:  &\dot{x}=4(a_{05}y+a_{40}x), ~~
\dot{y}=5b_{31}y, \\
&& ~{\it with}~ a_{05}a_{40}b_{31} \ne0, ~ {\it by}~ dt= dt_1/x^{\frac{3}{4}},~x>0,
 \\
 & \widetilde{X}_{015}:   &\dot{x}=5a_{05}, ~~
\dot{y}=6b_{40},  \\
&& ~{\it with}~ a_{05}b_{40} \ne0, ~ {\it by}~ dt= dt_1/(x^{\frac{4}{5}}y^{\frac{5}{6}}),~y>0,
 \\
& \widetilde{X}_{021}:  &\dot{x}=a_{05}y +a_{12}x, ~~
\dot{y}=3(b_{03}y+b_{10}x),\\
&& ~{\it with}~ a_{05}b_{10} \ne0, ~ {\it by}~ dt= dt_1/y^{\frac{2}{3}},
\\
& \widetilde{X}_{023}:  &\dot{x}=3(a_{05}y+a_{30}x), ~~
\dot{y}=5b_{21}y,\\
&&~{\it with}~ a_{05}a_{30}b_{21} \ne0, ~ {\it by}~ dt= dt_1/x^{\frac{2}{3}},
\\
\end{eqnarray*}
\begin{eqnarray*}
& \widetilde{X}_{032}:  &\dot{x}=2(a_{05}y+a_{20}x), ~~
\dot{y}=5b_{11}y,\\
&&~{\it with}~ a_{05}a_{20}b_{11} \ne0, ~ {\it by}~ dt= dt_1/x^{\frac{1}{2}}, ~x>0,
\\
& \widetilde{X}_{111}:  &\dot{x}=a_{14}xy^2 +a_{22}x^2y +a_{30}x^3, ~~
\dot{y}=2(b_{05}y^3+b_{13}xy^2+b_{21}x^2y),\\
&& ~{\it with}~ a_{30}b_{05} \ne0, ~y>0,
 \\
 & \widetilde{X}_{113}:  &\dot{x}=3x(a_{14}y +a_{40}x), ~~
\dot{y}=4y(b_{05}y+b_{31}x),\\
&& ~{\it with}~ a_{40}b_{05} \ne0, ~y>0,
 \\
 & \widetilde{X}_{114}:  &\dot{x}= 4a_{14}x, ~~
\dot{y}=5(b_{05}y+b_{40}x)
\\
&& ~{\it with}~ a_{14}b_{05}b_{40} \ne0,~{\it by}~ dt=dt_1 /y^{\frac{4}{5}},   ~x>0,
 \\
 & \widetilde{X}_{123}:  &\dot{x}=3a_{14}x, ~~
\dot{y}=5(b_{05}y+b_{30}x),\\
&&~{\it with}~ a_{14} b_{05}b_{30} \ne0, ~ {\it by}~ dt= dt_1/y^{\frac{4}{5}},
\\
& \widetilde{X}_{131}:  &\dot{x}=a_{14}xy +a_{20}x^2, ~~
\dot{y}=4(b_{05}y^2+b_{11}xy),
\\
&& ~{\it with}~ a_{20}b_{05} \ne0, ~y>0,
\\
& \widetilde{X}_{132}: &\dot{x}=2a_{14}x, ~~
\dot{y}=5(b_{05}y+b_{20}x),\\
&& ~{\it with}~ a_{14}b_{05}b_{20} \ne0, ~ {\it by}~ dt= dt_1/y^{\frac{4}{5}}, ~x>0,
\\
& \widetilde{X}_{141}: &\dot{x}=a_{14}x, ~~
\dot{y}=5(b_{05}y+b_{10}x),
\\
&& ~{\it with}~ a_{14}b_{05}b_{10} \ne0, ~ {\it by}~ dt= dt_1/y^{\frac{4}{5}},
\\
& \widetilde{X}_1:  ~ &\dot{x}=a_{05}y+a_{10}x , ~~ \dot{y}=5b_{01}y,
\\
&& ~{\it with}~ a_{05}a_{10}b_{01} \ne 0,
\end{eqnarray*}
}
where we still use $x, y$ replacing $\widetilde{x}, \widetilde{y}$ for simplifying notations.
Note that each $\widetilde{X}_{ijk}$ or $\widetilde{X}_1$ is one of the  four systems in the theorem. We complete the proof of the theorem.
\end{proof}

\bigskip


\section{Properties of homogeneous vector fields}
\label{s3}


For studying topological phase portraits of the quintic quasi--homogeneous but non--homogeneous systems via Theorem \ref{th-homo-1}, we need the knowledge on homogeneous polynomial differential system. Here we summarize some of them for the next dynamical analysis of quasi--homogeneous differential systems.

For a planar homogeneous polynomial vector field of degree $n\in \mathbb N\setminus\{1\}$, its origin is a highly degenerate singularity. For studying its local dynamics around the origin, the blow--up technique is useful. Commonly, we can blow up a degenerate singularity
in several less degenerate singularities either on a cycle or in a line. Here, we choose the latter, which can be applied to the singularities both in the finite plane and at the infinity.

Consider the planar homogeneous polynomial vector field of degree $n>1$
\begin{equation} \label{homo-n}
 \mathcal{H}_n:  \dot{x}=\sum_{i+j=n}~ c_{ij}x^iy^j:=P_n(x,y), ~~~ \dot{y}=\sum_{i+j=n}~ d_{ij}x^iy^j:=Q_n(x,y),
 \end{equation}
with $c_{ij}, d_{ij}\in\mathbb R$ not all equal to zero, and $P_n$ and $Q_n$ coprime.
The  origin is the unique singularity of system $\mathcal H_n$ in $\R^2$.
The change of variables
\begin{equation}
x=x, ~~ y=ux,
 \label{blow-up}
 \end{equation}
transforms system \eqref{homo-n} into
\begin{equation}
\begin{cases}
  \dfrac{dx}{d\tau} =x P_n(1,u)=x \sum_{i+j=n}~ c_{ij}u^j,
  \\
  \dfrac{du}{d\tau} =G(1, u)=\sum_{i+j=n}~ (d_{ij}u^j-c_{ij}u^{j+1} ),
\end{cases}
 \label{sys-xu}
 \end{equation}
after the time rescaling $dt=d\tau/x^{n-1}$, where
\begin{equation}
G(x,y):=x Q_n(x, y)-y P_n(x, y).
 \label{Gxy}
 \end{equation}

Recall that \eqref{blow-up} is the Briot--Bouquet transformation, see e.g. \cite{Bri1856}, which decomposes the origin of system \eqref{homo-n}
 into several simpler ones located on the $u$--axis of system \eqref{sys-xu}. That is, the zeros of $G(1, u)$ determine the equilibria located on the $u$-axis of system \eqref{sys-xu},
which correspond to the characteristic directions of system \eqref{homo-n} at the origin.

In \cite{Cima1990} Cima and Llibre presented a classification of the singularity at the origin of system \eqref{homo-n}. Here we will provide more detail conditions for characterizing the singularity.

Suppose that there exists a $u_0$ satisfying $G(1, u_0)=0$. Then $E=(0,u_0)$ is a singularity of system \eqref{sys-xu}
and system \eqref{homo-n} has the characteristic direction $\theta=\arctan (u_0)$ at the origin.
Besides, $P_n(1,u_0)\ne 0$ since $G(1,u_0)=0$. Otherwise, $y-u_0 x$ is a common factor of $Q_n(x, y)$ and $P_n(x, y)$, a contradiction with the fact that  $Q_n(x, y)$ and $P_n(x, y)$ are coprime. In the following we denote by $G'(1,u_0)$ the derivative with respect to $u$ at $u_0$.

\begin{proposition} \label{th-BU}
For the singularity $E=(0,u_0)$ of system \eqref{sys-xu}, the following statements hold.
\begin{itemize}
\item[$(a)$] For $G'(1,u_0)\ne 0$, the singularity $E$ is
\begin{itemize}
\item either a saddle if $P_n(1, u_0) G'(1,u_0)<0$,
\item or a node if $P_n(1, u_0) G'(1,u_0)>0$.
\end{itemize}
\item[$(b)$] For $G'(1,u_0)=0$, if $u_0$ is a zero of multiplicity $m>1$ of $G(1,u)$, then $E$ is
\begin{itemize}
\item  either a saddle if $m$ is odd and  $P_n(1, u_0) G^{(m)}(1,u_0)<0$,
\item or  a node if $m$ is odd and  $P_n(1, u_0) G^{(m)}(1,u_0)>0$,
\item  or a saddle--node if  $m$ is even.
\end{itemize}
\end{itemize}
\end{proposition}

 \begin{proof}
 An easy computation shows that the Jacobian matrix of  system  (\ref{sys-xu}) at singularity $E=(0,u_0)$ is
\begin{eqnarray*}
J(E)=\begin{pmatrix}
P_n(1, u_0) & 0\\
0 & G'(1,u_0)
\end{pmatrix}.
\end{eqnarray*}
Then statement $(a)$ follows from this expression of $J(E)$.

When $G(1,u_0)=G'(1,u_0)=0$, we get  $P_n(1,u_0)\ne 0$. In this case, the matrix $J(E)$ has the eigenvalues $0$ and $P_n(1,u_0)$.
Moving $E: (0,u_0)$ to the origin and taking the time rescaling $d\tau=d\tau_1/P_n(1, u_0)$,  system \eqref{sys-xu} can be rewritten as
 \begin{equation}  \label{sys-xu2}
\begin{cases}
\dfrac{dx}{d\tau_1} =x P_n(1,u+u_0)/P_n(1, u_0)=x \sum\limits_{i+j=n}~ c_{ij}(u+u_0)^j/P_n(1, u_0),
  \\
\dfrac{du}{d\tau_1} =G(1, u+u_0)/P_n(1, u_0)=\sum\limits_{i+j=n}~ (d_{ij}-c_{ij}(u+u_0) )(u+u_0)^j/P_n(1, u_0).
\end{cases}
 \end{equation}
Notice that $x=0$ is invariant under the flow of system \eqref{sys-xu2} and the linear part of the system is in canonical form.
Since $u_0$ is a zero of multiplicity $m$ of $G(1,u)$, we get from \cite[Theorem 7.1, Chapter 2]{ZDHD} the conclusion of statement $(b)$. This proves the proposition.
 \end{proof}

\medskip

Note that the change \eqref{blow-up} is singular on $x=0$. We now determine the number of orbits which
approach origin and are tangent to the $y$--axis.

Applying the polar coordinate changes $x=r \cos\theta$ and $y=r \sin\theta$, system \eqref{homo-n} can be written in
\begin{eqnarray}
 \frac{1}{r}\frac{dr}{d\theta}=\frac{\widetilde{H}(\theta)}{\widetilde{G}(\theta)},
 \label{H/G}
\end{eqnarray}
where
\begin{eqnarray*}
\widetilde{G}(\theta)&=&G(\cos\theta, \sin\theta),\\
\widetilde{H}(\theta)&=&H(\cos\theta, \sin\theta), ~~~~
\\
 H(x, y)& := & yQ_n(x,y)+xP_n(x,y).
\end{eqnarray*}
Hence a necessary condition for $\theta=\theta_0$ to be a characteristic  direction of the origin is $G(\cos\theta_0, \sin\theta_0)=0$.

If   $u_0$  is a root of the equation $G(1, u)=0$, then $H(1,u_0) \ne 0$ since
 \begin{eqnarray}
 H(1,u_0)=(1+u_0^2) P_n(1,u_0),
 \label{H=P}
\end{eqnarray}
and  $P_n(1,u_0)\ne 0$  from the fact that  $Q_n(x, y)$ and $P_n(x, y)$ are coprime.

The next proposition provides the properties of system \eqref{homo-n} at the origin along the characteristic  direction $\theta=\frac{\pi}{2}$,
i.e., $G(v, 1)=0$ at $v=0$, where $v=x/y$.

\medskip

\begin{proposition}
\label{th-y}
Assume that $\theta=\frac{\pi}{2}$ is a zero of multiplicity $m$ of $\widetilde{G}(\theta)$. The following statements hold.
\begin{itemize}
\item If $m>0$ is even, there exist infinitely many orbits connecting the origin of system \eqref{homo-n} and being tangent to the $y$--axis at the origin.
\item If $m$ is odd, there exist either infinitely many orbits if
$
\widetilde{G}^{(m)}(\frac{\pi}{2})\widetilde{H}(\frac{\pi}{2})> 0,
$
or exactly one orbit if
$
\widetilde{G}^{(m)}(\frac{\pi}{2})\widetilde{H}(\frac{\pi}{2})<0,
$
connecting the origin of  system \eqref{homo-n} and being tangent to the $y$--axis at the origin.
\end{itemize}
\end{proposition}

\begin{proof}
Since $\theta=\frac{\pi}{2}$ is a zero of multiplicity $m$ of $\widetilde G(\theta)$, we have $\widetilde{G}^{(m)}(\frac{\pi}{2})\ne 0$. In addition, we have
\[
\widetilde{H}(\frac{\pi}{2})= H(\cos\frac{\pi}{2}, \sin\frac{\pi}{2})=H(0, 1)\ne0,
\]
where we have used the facts that $G(0, 1)=0$ and that $Q_n(x, y)$ and $P_n(x, y)$ have no a common factor.

Applying the results on normal sectors, see e.g.  \cite[Theorems 4--6, Chapter 5]{Sans} or \cite[Theorems 3.4, 3.7 and 3.8, Chapter 2]{ZDHD},
we obtain that either  infinitely many orbits connect the origin of system \eqref{homo-n} along the characteristic direction $y$--axis if
 $m$ is even, or $m$ is odd and
$
\widetilde{G}^{(m)}(\frac{\pi}{2})\widetilde{H}(\frac{\pi}{2})> 0,
$
or exact one orbit connects the origin of  system \eqref{homo-n} along the direction $y$--axis
if $m$ is odd and
$
\widetilde{G}^{(m)}(\frac{\pi}{2})\widetilde{H}(\frac{\pi}{2})<0.
$
\end{proof}

\medskip


The next result, due to Cima and Llibre \cite{Cima1990}, will be used later on.

\begin{lemma}\label{lm-inf}
For the homogeneous  system \eqref{homo-n}, the following statements hold.
\begin{itemize}
\item The vanishing set of each linear factor $($if exists$)$ of $G(x,y)$ is an invariant line of the homogeneous  system \eqref{homo-n}.
\item Each singularity of system \eqref{homo-n} at infinity is either a node, or a saddle, or a saddle--nodes. The saddle--node happens if and only if
the corresponding root $u_0$ of the equation $G_n(1,u)=0$  is of even multiplicity.
\item The origin  of  system \eqref{homo-n}
is a global center if and only if
\[
\int_{-\infty}^{\infty} ~\frac{P_n(1,u)}{G_n(1,u)}~du=0.
\]
\end{itemize}
\end{lemma}

From the  expression of $G(x, y)$ in \eqref{Gxy}, we know that the origin of a homogeneous polynomial system cannot be a center if the degree of the system
is even. Actually, in this case there exists a  $u_0$ satisfying $G(1, u_0)=0$ and the system has an invariant line $y=u_0 x$.

\begin{proposition}
 A homogeneous polynomial system  has  a center at the origin only  if the degree of the system is odd.
\end{proposition}

\bigskip


\section{Global structure of quasi--homogeneous but non--homogeneous quintic  systems}

From Theorem \ref{th-homo-1}, there exist four homogeneous polynomial differential systems, which are translated from quintic quasi--homogeneous but non--homogeneous systems. Here we only study the global topological structures of quintic quasi--homogeneous differential systems, which are transformed into the homogeneous differential systems $\mathcal H_3$ and $\mathcal H_2$. We will not study the ones which are transformed into $\mathcal H_1$ and $\mathcal H_0$, because they are either linear or constant systems, whose structures are simple and so are omitted.

Note that separatrices of a quasi--homogeneous system are usually not easy to be determined directly. We will need the help of the ones of the related homogeneous differential systems.
After obtaining the phase portraits of the homogeneous systems $\mathcal H_3$ and $\mathcal H_2$,
we can get the phase portraits of the original quintic quasi--homogeneous systems through symmetry of the system
with respect to the axes and the change of coordinates \eqref{change2}.

First we study the homogeneous system $\mathcal{H}_3$, which corresponds to the quasi--homogeneous system $X_{111}$. Based on the classification of  fourth--order binary forms,
Cima and Llibre \cite{Cima1990} obtained the algebraic characteristics of cubic homogeneous  systems and further they
researched  all phase portraits of such canonical cubic homogeneous  systems. However, it is not easy to change a cubic homogeneous  system
to its canonical form since one needs to solve four quartic polynomial equations. We will apply Propositions \ref{th-BU} and \ref{th-y} and Lemma \ref{lm-inf} to obtain the global dynamics of  system $\mathcal{H}_3$ and consequently those of quasi--homogeneous system $X_{111}$.


\begin{theorem}
\label{th-X111}
The quintic quasi--homogeneous system $X_{111}$ has 52
topologically equivalent   global phase--portraits without taking into account the direction of the time.
\end{theorem}

\begin{proof}
Taking respectively the Poincar\'e transformations
$x=1/z, \, y=u/z$ and $x=v/z, \, y=1/z$  together with the time rescaling $d\tau_1=b_{05}dt/z^2$,
system  $X_{111}$  around the equator of the Poincar\'e sphere can be written respectively  in
\begin{eqnarray}
\begin{cases}
\dot{u}=u ( (b_{21}-a_{30})z^2 +(b_{13}-a_{22})u^2z + (1-a_{14})u^4 ),
\\
\dot{z}=-z(a_{30}z^2+a_{22}u^2z+a_{14}u^4).
\end{cases}
\label{X111-Ix}
\end{eqnarray}
and
\begin{eqnarray}
\begin{cases}
\dot{v}=v ( (a_{14}-1) +(a_{22}-b_{13})vz +(a_{30}-b_{21})v^2z^2 ),
\\
\dot{z}=-z(1+b_{13}vz+b_{21}v^2z^2),
\end{cases}
\label{X111-Iy}
\end{eqnarray}
where we still keep the notations of parameters $a_{ij}$ and $b_{ij}$ for simplicity.
Therefore, there exist only singularities $I_0$ and $I_1$ of system $X_{111}$ located at the infinity of the $x$-axis and
the $y$-axis respectively if $a_{14}\ne 1$. It is not hard to  see that $I_1$ is either a saddle if $a_{14}>1$,
or a stable node  if $a_{14}<1$, and $I_0$ is always  degenerate. When $a_{14}=1$, the infinity fulfils singularities.

Notice from Theorem \ref{th-invariant} that the corresponding quintic quasi--homogeneous system $X_{111}$ associated to $\mathcal{H}_3$ is invariant
under the action $(x, y)\to (x,-y)$ and the vector field associated to system $X_{111}$  is symmetric with respect to the $x$--axis.
We only need to study the dynamics of system $X_{111}$ in the half plane $y>0$, where the change \eqref{change2} is invertible.

We now investigate the global dynamics of system $\mathcal{H}_3$, and then obtain the ones of  system $X_{111}$.
After the time rescaling $dt=dt_1/d_{03}$, the cubic homogeneous system $\mathcal{H}_3$ in Theorem \ref{th-homo-1}
becomes
\begin{equation}
\label{H3PQ}
\begin{cases}
 \dot{x}=x(c_{12}y^2 +c_{21}xy +c_{30}x^2):=P_3(x, y),
 \\
\dot{y}=y(y^2+d_{12}xy+d_{21}x^2):=Q_3(x, y),
  \end{cases}
\end{equation}
where $c_{30}\ne 0$ and we keep the notations for the parameters $c_{ij}, d_{i j}$ for simplicity.
The change of variables \eqref{blow-up} transforms the cubic homogeneous system \eqref{H3PQ} into
\begin{equation}
\begin{cases}
  \dot{x}=x \widehat P_3(u):=x ( c_{12}u^2 +c_{21}u +c_{30} ),
  \\
   \dot{u}=\widehat G_3(u):=u( (1-c_{12})u^2 +(d_{12}-c_{21})u +d_{21}-c_{30}),
\end{cases}
 \label{sys-xu1}
 \end{equation}
where
\[
\widehat {P}_3(u)=P_3(1,u), ~~ \widehat {G}_3(u)=G_3(1,u), ~~G_3(x,y)= xQ_3(x,y)-y P_3(x,y).
\]
Suppose that $u_0$ is a zero of $\widehat {G}_3(u)$. Notice that the multiplicity of $u_0$ is at most 3 and $\widehat {G}_3(u)$ has at most three different zeros.

By \eqref{H/G} and \eqref{H=P}, we get from $H_3(x, y) = yQ_3(x,y)+xP_3(x,y)$ that
\begin{eqnarray}
\widehat {H}_3(u_0):= H_3(1,u_0)=(1+u_0^2) \widehat {P}_3(u_0)\ne 0,
 \label{H3}
\end{eqnarray}
which implies that the sign of  $\widehat {P}_3(u_0)$ determines the direction of the orbits along the line $y=u_0 x$.
If  $\widehat {P}_3(u_0)>0$ (resp. $<0$), the orbits will  leave from (resp. approach) the origin  along the characteristic direction $\theta=\arctan (u_0)$ in the positive time.

By Proposition \ref{th-BU} and Lemma \ref{lm-inf},
the singularity $E=(0,u_0)$ of system \eqref{sys-xu1} is  a saddle if either $\widehat {P}_3(u_0) \widehat {G}'_3(u_0)<0$,
or $\widehat {G}'_3(u_0)=\widehat {G}''_3(u_0)=0$ and $\widehat {P}_3(u_0) \widehat {G}^{'''}_3(u_0)<0$.
And $E=(0,u_0)$ is a node  if either $\widehat {P}_3( u_0) \widehat {G}'_3(u_0)>0$,
or $\widehat {G}'_3(u_0)=\widehat G''_3(u_0)=0$ and $\widehat {P}_3(u_0) \widehat {G}^{'''}_3(u_0)>0$.
These show that  except the invariant line $y=u_0 x$  system $\mathcal{H}_3$  has either no orbits or infinitely many orbits
connecting the origin along the characteristic directions $\theta=\arctan (u_0)$.

In contrast, if  $\widehat {G}'_3(u_0)=0$ and $\widehat {G}''_3(u_0)\ne 0$, we get from Proposition \ref{th-BU} that the singularity
$E=(0,u_0)$ of system \eqref{sys-xu1} is a saddle--node.
More precisely, there exist infinitely many orbits  of system $\mathcal{H}_3$  connecting the origin
along the direction of the invariant line $y=u_0 x$ if $u_0$ is a zero of multiplicity 2 of  $\widehat {G}_3(u)$.


Notice that $\theta=\frac{\pi}{2}$ is a zero of   $G_3(\cos\theta, \sin\theta)$  with multiplicity  $m\le 3$.
From Proposition \ref{th-y},
there exist infinitely many orbits (resp. exactly one orbit) connecting the origin of system \eqref{H3PQ} and being tangent to the $y$--axis at the origin either if $m=2$
or  $m$ is odd and
$
\widetilde{G}^{(m)}(\frac{\pi}{2})\widetilde{H}(\frac{\pi}{2})> 0
$
(resp.   if $m$ is odd and
$
\widetilde{G}^{(m)}(\frac{\pi}{2})\widetilde{H}(\frac{\pi}{2})<0).
$


From Lemma \ref{lm-inf},
each singularity of system  \eqref{H3PQ}  at infinity is either a node, or a saddle or a saddle--node.
In order to get the concrete parameter conditions determining the dynamics of system  \eqref{H3PQ} at infinity,
we need the Poincar\'e compactification \cite{Dumor2006}.

Taking respectively the Poincar\'e transformations
$x=1/z, \, y=u/z$ and $x=v/z, \, y=1/z$  together with the time rescaling $d\tau=dt_1/z^2$,
system  \eqref{H3PQ} around the equator of the Poincar\'e sphere can be written respectively in
\begin{eqnarray}
\dot{u}= G_3(1,u), ~~~ \dot{z} = -zP_3(1, u),
  \label{sys-Ix}
\end{eqnarray}
and
\begin{eqnarray}
\dot{v}= -G_3(v,1), ~~~ \dot{z} = -zQ_3(v,1).
\label{sys-Iy}
\end{eqnarray}
Therefore, a singularity $I_{u_0}$ of system $\mathcal{H}_3$ located at the infinity of the line $y=xu_0 $
is either a saddle if $\widehat {P}_3( u_0) \widehat {G}'_3(u_0)>0$,  or $\widehat {G}'_3(u_0)=\widehat {G}''_3(u_0)=0$
and $\widehat {P}_3(u_0) \widehat {G}^{(3)}_3(u_0)>0$;
or a node  if $\widehat {P}_3( u_0) \widehat {G}'_3(u_0)<0$, or $\widehat {G}'_3(u_0)=\widehat {G}''_3(u_0)=0$ and $\widehat {P}_3(u_0) \widehat {G}^{(3)}_3(u_0)<0$.
If  $\widehat {G}'_3(u_0)=0$ and $\widehat {G}''_3(u_0)\ne 0$, the singularity $I_{u_0}$  at infinity  is a saddle--node.

System  \eqref{H3PQ} has a singularity  $I_y$ located at the end of the
$y$--axis. It is a saddle if $c_{12}>1$,
or $c_{12}=1$, $d_{12}=c_{21}$ and $d_{21}<c_{30}$.
The singularity $I_y$ is a stable node if $c_{12}<1$,
or $c_{12}=1$, $d_{12}=c_{21}$ and $d_{21}>c_{30}$. And it is a saddle--node if  $c_{12}=1$ and $d_{12}\ne c_{21}$.

Summarizing the above analysis, one needs to distinguish three cases:
\begin{itemize}
\item[$(i)$]  $\widehat {G}_3(u)$ has three different real zeros $0$ and $u_{\pm} :=\dfrac{c_{21}-d_{12}\pm \sqrt{\Delta}}{2(1-c_{12})}$ if
$
\Delta>0, ~~ d_{21}\ne c_{30}, $ and $c_{12} \ne 1,
$
where $\Delta:=(d_{12}-c_{21})^2-4(1-c_{12})(d_{21}-c_{30})$.
\item[$(ii)$] $\widehat {G}_3(u)$ has two different real zeros $0$ and $u_1$,
where
\begin{equation*}
u_1=
\begin{cases}
  u_{11},  &~~\mbox{if}~ d_{21}= c_{30}, ~(d_{12}-c_{21})(c_{12}-1)\ne0,
  \\
  u_{12}, &~~\mbox{if}~ c_{12}=1, ~(c_{30}-d_{21})(d_{12}-c_{21})\ne0,
  \\
  u_{13}, &~~\mbox{if}~ \Delta=0, c_{21}\ne d_{12},
\end{cases}
  \end{equation*}
with $ u_{11}=\frac{d_{12}-c_{21}}{c_{12}-1}$, $u_{12}=\frac{c_{30}-d_{21}}{d_{12}-c_{21}}$ and $u_{13}=\frac{c_{21}-d_{12} }{2(1-c_{12})}$.
\item[$(iii)$] $\widehat {G}_3(u)$ has only one real zero  $0$ if either ($\mathcal{C}_{31}$): $\Delta<0$  and $(c_{12}- 1)(d_{21}-c_{30}) \ne 0$,
or ($\mathcal{C}_{32}$): $c_{12}= 1$ and $d_{12}=c_{21}$,
 or ($\mathcal{C}_{33}$): $d_{21}= c_{30}$ and $(c_{12}- 1)(d_{12}-c_{21}) = 0$ are satisfied.
 \end{itemize}
Notice that $P_3$ and $Q_3$ share a non--constant common factor if $d_{21}=c_{30}$, $c_{12}= 1$ and $d_{12}=c_{21}$, which is not under the consideration.

Going back to the original system  $X_{111}$,
 the invariant line $y=u_0x$ of system $\mathcal H_3$  as $u_0\ne 0$ is an invariant curve of system  $X_{111}$,
which is tangent to the $y$-axis at the origin   and  connects the origin and the singularity $I_0$  at infinity.
Moreover, the invariant curve of system  $X_{111}$ is usually a separatrix of hyperbolic sectors, parabolic sectors or elliptic sectors, which is drawn in
thick blue curves in
Fig. \ref{Fig_global-2}.
The above analyses provide enough preparation for studying global topological phase portraits of the quintic quasi-homogeneous system $X_{111}$.
By the properties of the singularities $I_1$ at infinity, we discuss three cases: $a_{14}>1$, $a_{14}<1$ and $a_{14}=1$.

In the case $a_{14}<1$, we have that $c_{12}<1/2<1$  and $I_1$ is  a stable  node.
We discuss  three subcases $(i)$--$(iii)$ under the condition  $c_{12}<1$.
In subcase $(i)$, we obtain the $6$ global phase portraits shown in Fig. \ref{Fig_global-2} ($\mathcal{I}$)-- ($\mathcal{V}$),
in which we let $u_+>0>u_-$ for the corresponding cubic homogeneous system $\mathcal H_3$.
In subcase $(ii)$, we get the $6$ global phase portraits for the system  $X_{111}$ given in Fig.
\ref{Fig_global-2}  ($\mathcal{VI}$)--($\mathcal{VIII}$.3), where we choose $u_1>0$
for the corresponding cubic homogeneous system $\mathcal H_3$.
In subcase $(iii)$, we obtain the $2$ global  phase portraits for the system  $X_{111}$
given in Fig. \ref{Fig_global-2} ($\mathcal{IX}$)--($\mathcal{X}$).


\begin{figure}[htp]
\begin{center}
\begin{tabular}{cc}
\includegraphics[width=4in]{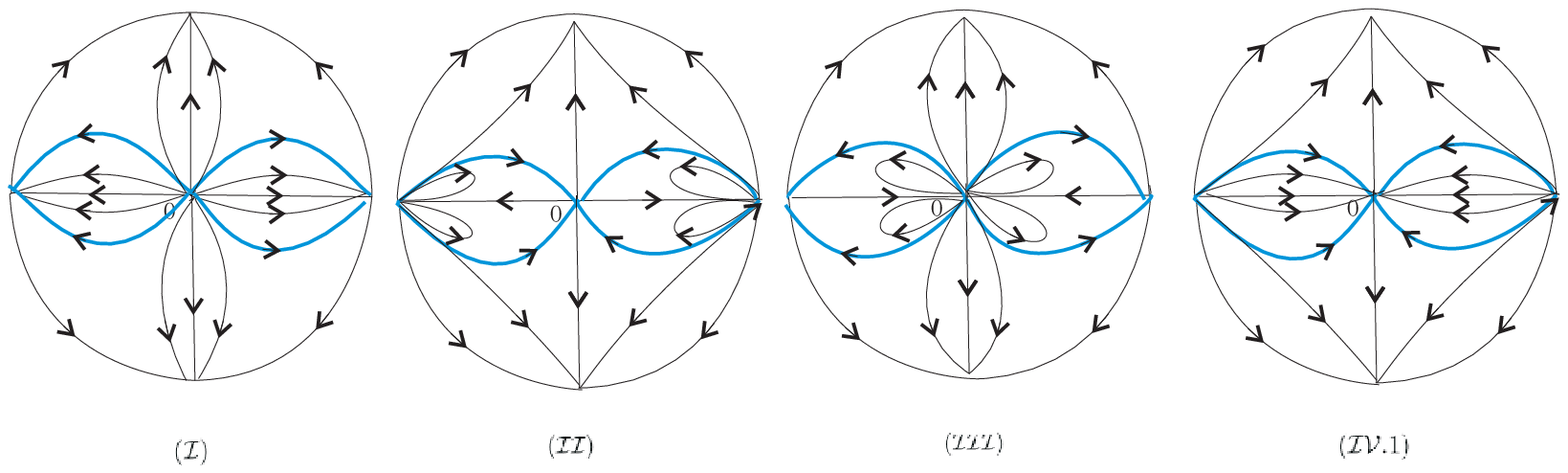}
\\
\includegraphics[width=4in]{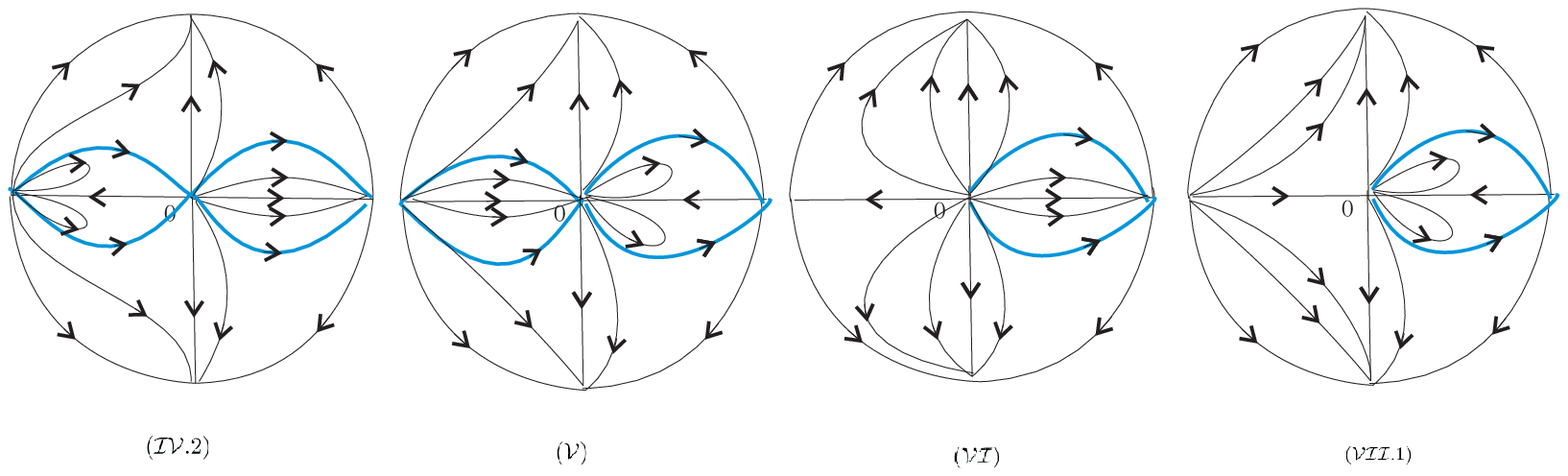}
\end{tabular}
\end{center}


\begin{center}
\begin{tabular}{cc}
\includegraphics[width=4in]{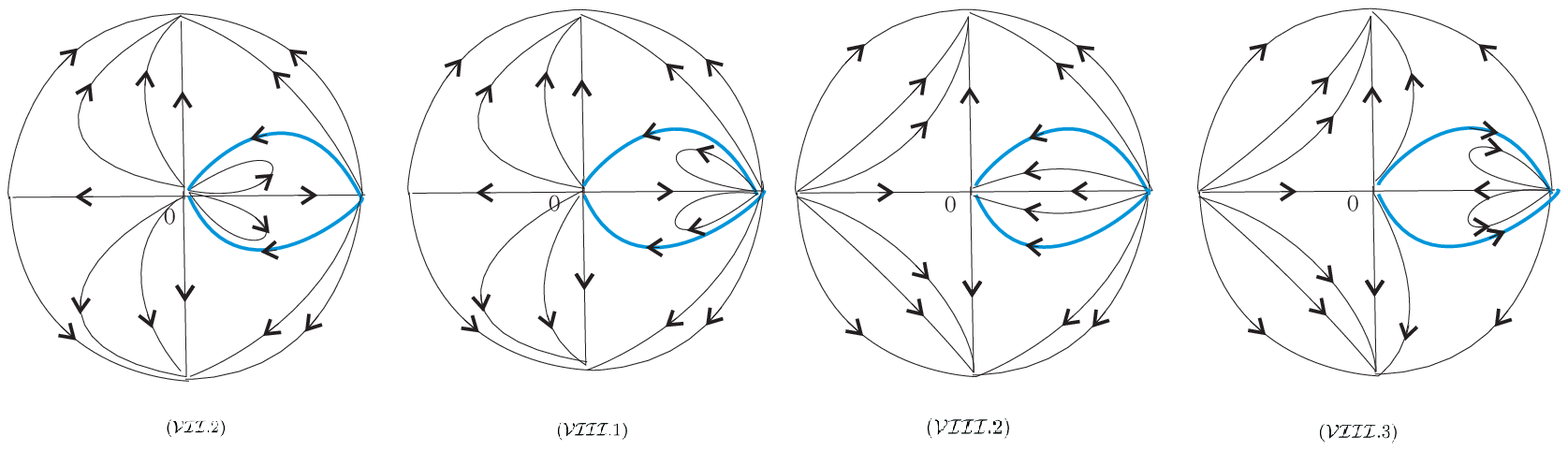}
\\
\includegraphics[width=2in]{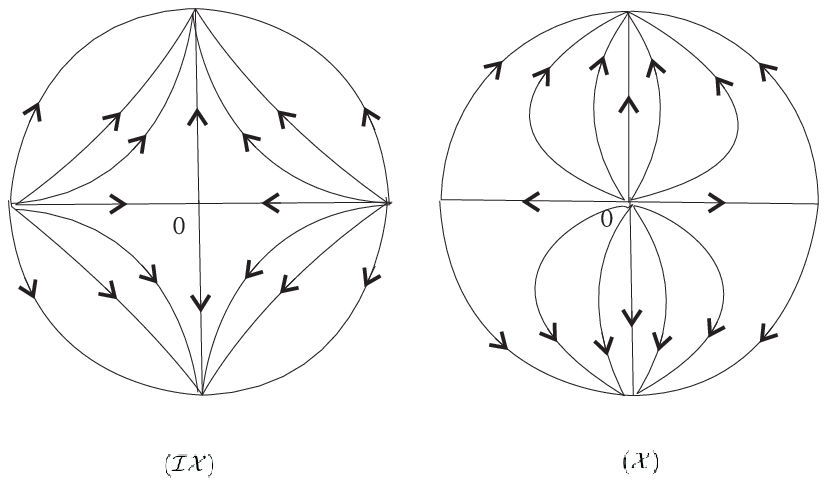}
\end{tabular}
\end{center}
\caption{Global phase portraits of system $X_{111}$ as $a_{14}<1$. }
\label{Fig_global-2}
\end{figure}

In the case $a_{14}>1$,  equilibrium $I_1$ is a saddle
and we obtain the $24$ global  phase portraits for the system  $X_{111}$ by a similar research as the case $a_{14}<1$.

%
%
%
%
%
%
%
%
%
%
%
%

Tables 1 and 2 show the parameter conditions under which system $X_{111}$
has the phase portraits, 
 where we have used the relation between the parameters of $X_{111}$ and $\mathcal H_3$, i.e. $c_{12}=a_{14}/(2b_{05})$,
$c_{21}=a_{22}/(2b_{05})$,  $c_{30}=a_{30}/(2b_{05})$, $d_{12}=b_{13}/b_{05}$, $d_{21}=b_{21}/b_{05}$.
Note that each upper half part of the  figure of  system $\mathcal H_3$ 
is equivalent to some half part of the one in Figure 5.1 (1)--(14) of Cima and Llibre  \cite{Cima1990}, which is also pointed out in Table 1 and Table 2.


{\footnotesize
\begin{tabular}{|c|c|c|}
\hline
Figure & parameter conditions & Figure 5.1 of \cite{Cima1990}
\\
\hline
Figure (I) & $\widehat{P}_3(u_{\pm})>0$,  $\widehat{G}'_3(u_+)\widehat{G}'_3(u_-)>0$, $\widehat{P}_3(0)>0$ & Figure (4)
\\
\hline
Figure (I)  & \mbox{or} $u_1=u_{11}$, $\widehat{P}_3(u_1)>0$,  $\widehat{G}'_3(u_1)\ne 0$, $\widehat{P}_3(0)>0$, &
\\
 & \mbox{or} $u_1=u_{12}$, $\widehat{P}_3(u_1)>0$,  $\widehat{G}'_3(u_1)\ne 0$, $\widehat{P}_3(0)>0$, & Figure (8)
\\
& \mbox{or} $u_1=u_{13}$, $\widehat{P}_3(u_1)>0$,  $\widehat{G}''_3(u_1)\ne 0$, $\widehat{P}_3(0)>0$ &
\\
\hline
Figure (I)& \mbox{Condition} ($\mathcal{C}_{31}$) \mbox{or} ($\mathcal{C}_{32}$),  $\widehat{G}'_3(0)\ne0$, $\widehat{P}_3(0)>0$, &Figure (11)
\\
& \mbox{or} \mbox{Condition} ($\mathcal{C}_{33}$),  $\widehat{G}'''_3(0)\ne 0$, $\widehat{P}_3(0)>0$ &
\\
\hline
Figure (I)& \mbox{Condition} ($\mathcal{C}_{33}$),  $\widehat{G}''_3(0)\ne 0$, $\widehat{P}_3(0)>0$ & Figure (13)
 \\
 \hline
Figure (II) & $\widehat{P}_3(u_{\pm})<0$,  $\widehat{G}'_3(u_{\pm})>0$, $\widehat{P}_3(0)>0$ & Figure (1)
\\
\hline
Figure (III.1) & $\widehat{P}_3(u_{\pm})>0$,  $\widehat{G}'_3(u_{\pm})>0$, $\widehat{P}_3(0)<0$ & Figure (5)
\\
\hline
Figure (III.2) & $\widehat{P}_3(u_+)\widehat{P}_3(u_-)<0$,  $\widehat{G}'_3(u_{\pm})<0$, $\widehat{P}_3(0)>0$ & Figure (5)
\\
\hline
Figure (III.2)& \mbox{or}  $u_1=u_{11}$, $\widehat{P}_3(u_1)<0$,  $\widehat{G}'_3(u_1)<0$, $\widehat{P}_3(0)>0$,&
\\
& \mbox{or} $u_1=u_{12}$, $\widehat{P}_3(u_1)<0$,  $\widehat{G}'_3(u_1)<0$, $\widehat{P}_3(0)>0$, & Figure (9)
\\
\hline
Figure (III.3) &$\widehat{P}_3(u_{\pm})<0$,  $\widehat{G}'_3(u_{\pm})<0$, $\widehat{P}_3(0)<0$  & Figure (5)
\\
\hline
Figure (IV.1) & $\widehat{P}_3(u_{\pm})>0$,  $\widehat{G}'_3(u_{\pm})<0$, $\widehat{P}_3(0)<0$ & Figure (2)
\\
\hline
Figure (IV.2) & $\widehat{P}_3(u_+)\widehat{P}_3(u_-)<0$,  $\widehat{G}'_3(u_{\pm})>0$, $\widehat{P}_3(0)>0$ & Figure (2)
\\
\hline
Figure (IV.2)& \mbox{or}  $u=u_{11}$, $\widehat{P}_3(u_1)<0$,  $\widehat{G}'_3(u_1)>0$,  $\widehat{P}_3(0)>0$  & Figure (6)
\\
\hline
Figure (IV.3)   &  $\widehat{P}_3(u_{\pm})<0$,  $\widehat{G}'_3(u_{\pm})>0$, $\widehat{P}_3(0)<0$ & Figure (2)
\\
\hline
Figure (V.1) & $\widehat{P}_3(u_+)\widehat{P}_3(u_-)<0$,  $\widehat{G}'_3(u_{\pm})>0$, $\widehat{P}_3(0)<0$ & Figure (3)
\\
\hline
Figure (V.2) & $\widehat{P}_3(u_+)\widehat{P}_3(u_-)<0$,  $\widehat{G}'_3(u_{\pm})<0$, $\widehat{P}_3(0)<0$ & Figure (3)
\\
\hline
Figure (VI.1) & $u_1=u_{11}$, $\widehat{P}_3(u_1)>0$,  $\widehat{G}'_3(u_1)> 0$, $\widehat{P}_3(0)<0$, & Figure (7)
\\
\hline
Figure (VI.2) & $u_1=u_{12}$, $\widehat{P}_3(u_1)<0$,  $\widehat{G}'_3(u_1)< 0$, $\widehat{P}_3(0)<0$ & Figure (7)
\\
\hline
Figure (VI.3) & $u_1=u_{13}$, $\widehat{P}_3(u_1)<0$,  $\widehat{G}''_3(u_1)> 0$, $\widehat{P}_3(0)>0$ & Figure (7)
\\
\hline
Figure (VI.4) & $u_1=u_{13}$, $\widehat{P}_3(u_1)<0$,  $\widehat{G}''_3(u_1)< 0$, $\widehat{P}_3(0)>0$ & Figure (7)
\\
\hline
Figure (VI.5) & $u_1=u_{11}$, $\widehat{P}_3(u_1)>0$,  $\widehat{G}'_3(u_1)< 0$, $\widehat{P}_3(0)<0$, & Figure (7)
\\
\hline
Figure (VI.6) &  $u_1=u_{12}$, $\widehat{P}_3(u_1)<0$,  $\widehat{G}'_3(u_1)> 0$, $\widehat{P}_3(0)<0$ & Figure (7)
\\
\hline
Figure (VII.1) & $u_1=u_{12}$, $\widehat{P}_3(u_1)>0$,  $\widehat{G}'_3(u_1)<0$, $\widehat{P}_3(0)<0$, & Figure (6)
\\
& \mbox{or} $u_1=u_{13}$, $\widehat{P}_3(u_1)>0$,  $\widehat{G}''_3(u_1)>0$, $\widehat{P}_3(0)<0$
\\
\hline
Figure (VII.2) & $u_1=u_{12}$, $\widehat{P}_3(u_1)<0$,  $\widehat{G}'_3(u_1)>0$, $\widehat{P}_3(0)>0$ & Figure (6)
\\
\hline
Figure (VII.3) & $u_1=u_{11}$, $\widehat{P}_3(u_1)<0$,  $\widehat{G}'_3(u_1)>0$,  $\widehat{P}_3(0)<0$,  & Figure (6)
\\
& \mbox{or} $u_1=u_{13}$, $\widehat{P}_3(u_1)<0$,  $\widehat{G}''_3(u_1)> 0$, $\widehat{P}_3(0)<0$ &
\\
\hline
Figure (VIII.1) & $u_1=u_{12}$, $\widehat{P}_3(u_1)>0$,  $\widehat{G}'_3(u_1)>0$, $\widehat{P}_3(0)<0$, & Figure (9)
\\
& \mbox{or} $u_1=u_{13}$, $\widehat{P}_3(u_1)>0$,  $\widehat{G}''_3(u_1)< 0$, $\widehat{P}_3(0)<0$ &
\\
\hline
Figure (VIII.2) & $u_1=u_{11}$, $\widehat{P}_3(u_1)<0$,  $\widehat{G}'_3(u_1)<0$, $\widehat{P}_3(0)<0$,  &Figure (9)
\\
& \mbox{or} $u_1=u_{13}$, $\widehat{P}_3(u_1)<0$,  $\widehat{G}''_3(u_1)< 0$, $\widehat{P}_3(0)<0$ &
\\
\hline
Figure (IX) & \mbox{Condition} ($\mathcal{C}_{31}$), $\widehat{G}'_3(0)>0$, $\widehat{P}_3(0)<0$, &
\\
& \mbox{or} \mbox{Condition} ($\mathcal{C}_{32}$),  $\widehat{G}'_3(0)>0$, $\widehat{P}_3(0)<0$, & Figure (10)
\\
& \mbox{or} \mbox{Condition} ($\mathcal{C}_{33}$),  $\widehat{G}'''_3(0)>0$, $\widehat{P}_3(0)<0$ &
\\
\hline
Figure (X) & \mbox{Condition} ($\mathcal{C}_{31}$) \mbox{or} ($\mathcal{C}_{32}$),  $\widehat{G}'_3(0)<0$, $\widehat{P}_3(0)<0$, & Figure (12)
\\
& \mbox{or} \mbox{Condition} ($\mathcal{C}_{33}$),  $\widehat{G}'''_3(0)< 0$, $\widehat{P}_3(0)<0$ &
\\
\hline
Figure (XI) &  \mbox{Condition} ($\mathcal{C}_{33}$),  $\widehat{G}''_3(0)\ne 0$, $\widehat{P}_3(0)<0$ & Figure (14)
\\
\hline
\end{tabular}
}

 \hskip 1cm Table 1. Parameter conditions of figures (I)-(XI) as $a_{14}>1$.

\medskip


 {\footnotesize
\begin{tabular}{|c|c|c|}
\hline
Figure & parameter conditions & Figure 5.1 of \cite{Cima1990}
\\
\hline
Figure ($\mathcal{I}$) & $\widehat{P}_3(u_{\pm})>0$,  $\widehat{G}'_3(u_{\pm})>0$,  $\widehat{P}_3(0)>0$ & Figure (4)
\\
\hline
Figure ($\mathcal{II}$) & $\widehat{P}_3(u_{\pm})<0$,  $\widehat{G}'_3(u_{\pm})>0$, $\widehat{P}_3(0)>0$ & Figure (1)
\\
\hline
Figure ($\mathcal{III}$) & $\widehat{P}_3(u_{\pm})>0$,  $\widehat{G}'_3(u_{\pm})>0$, $\widehat{P}_3(0)<0$ & Figure (5)
\\
\hline
Figure ($\mathcal{IV}$.1) & $\widehat{P}_3(u_{\pm})<0$,  $\widehat{G}'_3(u_{\pm})>0$, $\widehat{P}_3(0)<0$ & Figure (2)
\\
\hline
Figure ($\mathcal{IV}$.2) & $\widehat{P}_3(u_+)\widehat{P}_3(u_-)<0$,  $\widehat{G}'_3(u_{\pm})>0$, $\widehat{P}_3(0)>0$ & Figure (2)
\\
\hline
Figure ($\mathcal{V}$) & $\widehat{P}_3(u_+)\widehat{P}_3(u_-)<0$,  $\widehat{G}'_3(u_{\pm})>0$, $\widehat{P}_3(0)<0$ & Figure (3)
\\
\hline
Figure ($\mathcal{VI}$)  &  $u_1=u_{11}$, $\widehat{P}_3(u_1)>0$,  $\widehat{G}'_3(u_1)>0$, $\widehat{P}_3(0)>0$, & Figure (8)
\\
& \mbox{or} $u_1=u_{13}$, $\widehat{P}_3(u_1)>0$,  $\widehat{G}''_3(u_1)> 0$, $\widehat{P}_3(0)>0$ &
\\
\hline
Figure ($\mathcal{VII}$.1) & $u_1=u_{11}$, $\widehat{P}_3(u_1)>0$,  $\widehat{G}'_3(u_1)> 0$, $\widehat{P}_3(0)<0$, & Figure (7)
\\
\hline
Figure ($\mathcal{VII}$.2)  & $u_1=u_{13}$, $\widehat{P}_3(u_1)<0$,  $\widehat{G}''_3(u_1)> 0$, $\widehat{P}_3(0)>0$ & Figure (7)
\\
\hline
Figure ($\mathcal{VIII}$.1) & $u_1=u_{11}$, $\widehat{P}_3(u_1)<0$,  $\widehat{G}'_3(u_1)>0$,  $\widehat{P}_3(0)>0$& Figure (6)
\\
\hline
Figure ($\mathcal{VIII}$.2) & $u_1=u_{11}$, $\widehat{P}_3(u_1)<0$,  $\widehat{G}'_3(u_1)>0$,  $\widehat{P}_3(0)<0$, & Figure (6)
\\
& \mbox{or} $u_1=u_{13}$, $\widehat{P}_3(u_1)<0$,  $\widehat{G}''_3(u_1)> 0$, $\widehat{P}_3(0)<0$ &
\\
\hline
Figure ($\mathcal{VIII}$.3) & $u_1=u_{13}$, $\widehat{P}_3(u_1)>0$,  $\widehat{G}''_3(u_1)>0$, $\widehat{P}_3(0)<0$ & Figure (6)
\\
\hline
Figure ($\mathcal{IX}$) & \mbox{Condition} ($\mathcal{C}_{33}$),  $\widehat{G}'''_3(0)> 0$, $\widehat{P}_3(0)<0$ & Figure (10)
\\
&\mbox{or} \mbox{Condition} ($\mathcal{C}_{31}$),  $\widehat{G}'_3(0)>0$, $\widehat{P}_3(0)< 0$  &
\\
\hline
Figure ($\mathcal{X}$)& \mbox{Condition} ($\mathcal{C}_{31}$),  $\widehat{G}'_3(0)>0$, $\widehat{P}_3(0)> 0$, & Figure (11)
\\
& \mbox{or} \mbox{Condition} ($\mathcal{C}_{33}$),  $\widehat{G}'''_3(0)> 0$, $\widehat{P}_3(0)>0$ &
\\
\hline
\end{tabular}
}

 \hskip 1cm Table 2. Parameter conditions of figures ($\mathcal{I}$)-($\mathcal{X}$) as $a_{14}<1$.

\bigskip

 At last, we consider the case $a_{14}=1$, which implies that $c_{12}=1/2 $  and the infinity fulfils singularities.
Besides, by a time rescaling $d\tau_2=z d\tau_1 $, system \eqref{X111-Iy} becomes
\begin{eqnarray}
\begin{cases}
\dot{v}=v (   (a_{22}-b_{13})v  +(a_{30}-b_{21})v^2z  ) ,
\\
\dot{z}=- (1+b_{13}v z +b_{21}v^2z^2 ) .
\end{cases}
\label{X111-Iy2}
\end{eqnarray}
Clearly,  system \eqref{X111-Iy2}  has no singularities   on the axis $z=0$.
 Note that the global structures of system \eqref{X111-Iy2} and  system \eqref{X111-Iy} are  equivalent outside $z=0$. Hence,
except the $y$--axis no orbits connect the singularities at infinity of system $X_{111}$ other than $I_0$.
Similar to the case $a_{14}<1$, we also have exactly the 14 subcases given in Table 2, and consequently we will obtain 14 different topological phase portraits, 
which are different from those in the case $a_{14}>1$ 
only now the infinity fulfils singularities instead of the invariant line at infinity.

Summarizing the above analysis, we get that system $X_{111}$ has totally 52 topological phase portraits
without taking into account direction of the time.
We complete the proof of the theorem.
 \end{proof}

 \medskip


Finally we study topological structures of the homogeneous system $\mathcal{H}_2$, which corresponds to the quasi--homogeneous systems $X_{011}$, $X_{113}$ and  $X_{131}$.

\begin{theorem}
\label{th-H2}
The global phase portraits of the quintic quasi--homogeneous systems $X_{011}$, $X_{113}$ and  $X_{131}$
are topologically equivalent to one of the phase portraits in Figs. \ref{Fig_global-3} and \ref{Fig_global-4}
without taking into account the direction of the time.
\end{theorem}

\begin{proof}

Since the homogeneous systems associated to $X_{113}$ and $X_{131}$ are sub--systems of the homogeneous systems associated to $X_{011}$
with $a_{05}=b_{20}=0$, we only need to study system $X_{011}$ for arbitrary $a_{05}$ and $b_{20}$.


Similar to the discussion of the cubic homogeneous system $\mathcal{H}_3$ in Theorem \ref{th-X111}, we set
\[
\mathcal{H}_2: =(P_2(x, y), Q_2(x, y))  \quad \mbox{and}\quad  G_2(x,y)= xQ_2(x,y)-y P_2(x,y).
\]
Through the change of variables between the quasi-homogeneous differential systems and their associated homogeneous ones we get that the invariant line $y=u_0x$ of system $\mathcal H_2$  as $u_0\ne 0$ corresponds to an invariant curve of system  $X_{011}$,
which is tangent to the $y$--axis at the origin, and connects the origin and the singularity at infinity located at the end of $x$-axis.
Moreover, this invariant curve of system  $X_{011}$ is usually a separatrix of hyperbolic sectors, parabolic sectors or elliptic sectors, which is drawn in thick blue curves in Fig. \ref{Fig_global-3}. 
Note that $G_2(x,y)$ is of degree three.
It follows that $X_{011}$ has at least one   characteristic  direction at the origin. 
From \cite[Chapter 10]{Ye1986}, we can move the  characteristic  direction to the $x$--axis and system $\mathcal{H}_2$ can be written by a linear transformation  as
\begin{eqnarray}
\begin{cases}
\dot{x}_1= \alpha_{11} x_1^2+\alpha_{12} x_1y_1+\alpha_{22} y_1^2,
\\
\dot{y}_1= \beta_{12}x_1y_1+\beta_{22} y_1^2.
\end{cases}
\label{H2-2}
\end{eqnarray}
Since system $\mathcal{H}_2$ has no a common factor, and so is system \eqref{H2-2}. It forces that $\alpha_{11}\ne 0$. Using the inverse of change \eqref{change2}, 
we transform system \eqref{H2-2} into
\begin{eqnarray}
\begin{cases}
\dot{x}= \alpha_{11} x^2+\alpha_{12} xy^2+\alpha_{22} y^4,
\\
\dot{y}= (\beta_{12}xy+\beta_{22} y^3)/2.
\end{cases}
\label{X011-2}
\end{eqnarray}
Taking  the Poincar\'e transformations
 $x=v/z, \, y=1/z$  together with the time rescaling $dt_1= z^3 dt $,
system  \eqref{X011-2}  around the equator of the Poincar\'e sphere can be written   in
\begin{eqnarray}
\begin{cases}
\dot{v}=\alpha_{22} +(\alpha_{12}-\frac{\beta_{22}}{2})vz +(\alpha_{11}-\frac{\beta_{12}}{2})v^2z^2,
\\
\dot{z}=-z^2(\frac{\beta_{22}}{2}+\frac{\beta_{12}}{2}vz) .
\end{cases}
\label{X011-Iy}
\end{eqnarray}
Therefore,  system \eqref{X011-2} has a unique singularity at the infinity when $\alpha_{22}\ne 0$, which is $\tilde{I}_0$ located at the end of the $x$--axis.  When $\alpha_{22}= 0$, the infinity of system \eqref{X011-2} is fulled up with singularities. Moving away the common factor $z$ of system \eqref{X011-Iy}, the new system has a unique singularity on $z=0$, i.e. $\widetilde I_1$, which is  at the end of the $y$--axis. For this new system we can check that $\tilde{I}_1$ is either a saddle if
$(2\alpha_{12}-\beta_{22})\beta_{22}>0$,  or a node  if $(2\alpha_{12}-\beta_{22})\beta_{22}<0$.
Moreover, under the condition $\alpha_{22}= 0$  we can show that system \eqref{X011-2} has always the singularity $\tilde{I}_0$ at the end of the $x$--axis, which is degenerate.
When $ \beta_{22}=0$ and $\alpha_{22}= 0$, system \eqref{H2-2} has a common factor, which is out of our consideration.
When $2\alpha_{12}-\beta_{22} =0$, $\beta_{22}\ne 0$ and $\alpha_{22}= 0$, moving away the common factor $z^2$ of system \eqref{X011-Iy},
the new system has no singularities on $z=0$.

Having the above preparation, we now study the global phase portraits of system   $X_{011}$ according to the number of the zeros of  $\widehat G_2(u):=G_2(1,u)$,
which can have either one, or two, or three real roots. The next results can be obtained by using \cite[Theorem 10.5]{Ye1986}, the details are omitted.

\smallskip

\noindent {\it Case }$(i)$. {\it $\widehat G_2(u)$   has three different real zeros}.
Then we obtain the $3$ global  phase portraits given in Fig.~\ref{Fig_global-3} for the system $X_{011}$.


\begin{figure}[htp]
\begin{center}
\begin{tabular}{cc}
\includegraphics[width=3in]{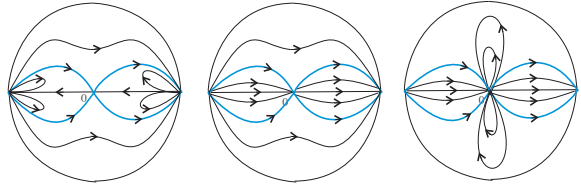}
\end{tabular}
\end{center}
\caption{Phase portraits of $X_{011}$ in case $(i)$. }
\label{Fig_global-3}
\end{figure}


\begin{figure}[htp]
\begin{center}
\begin{tabular}{cc}
\includegraphics[width=4in]{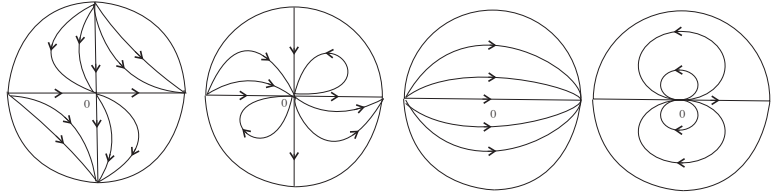}
\end{tabular}
\end{center}
\caption{Phase portraits of $X_{011}$ in case $(ii)$ and  case $(iii)$. }
\label{Fig_global-4}
\end{figure}

\noindent {\it Case }$(ii)$. {\it  $\widehat G_2(u)$ has  two  different  real zeros}. 
Then we get $2$ global phase portraits, which are given in the first two figures of Fig.~\ref{Fig_global-4} for  system $X_{011}$.

\noindent {\it Case }$(iii)$. {\it $\widehat G_2(u)$  has only one real zero}. 
Then we obtain  $2$ global  phase portraits, which are given in the last two figures of Fig.~\ref{Fig_global-4} for system $X_{011}$.

We complete the proof of the theorem.
 \end{proof}


We remark that the quintic quasi--homogeneous system other than $X_{111}$, $X_{011}$,  $X_{113}$ and  $X_{131}$ either has a nondegenerate linear part, or can be reduced to a linear system, or to a constant system by using Lemma \ref{lm-q5list} and Theorem \ref{th-homo}. Hence, their global structures are not difficult to be obtained.  The details are omitted here.

By Theorem \ref{th-homo} and \ref{th-homo-1} we can prove easily the results of \cite[Theorem 1.2]{TWZ2015}. In fact, only system $X_{021}$ in all quintic quasi--homogeneous but non--homogeneous systems can have a center at the origin when $a_{12}=-3b_{03}$ and $a_{12}b_{03}-a_{05}b_{10}>0$. And only in this case, the linear system $\widetilde{X}_{021}$ translated from $X_{021}$ and given in the proof of Theorem \ref{th-homo-1} has a pair of conjugate pure imaginary eigenvalues.


\section*{Acknowledgments}

The first author has received funding from the European Union's Horizon 2020 research and innovation
programme under the Marie Sklodowska-Curie grant agreement No 655212, and is  partially supported by the National Natural
Science Foundation of China (No. 11431008) and the RFDP of Higher Education of China grant (No.20130073110074).
The second author is partially supported by NNSF of China grant number 11271252,  by innovation program of Shanghai Municipal Education Commission grant 15ZZ012, and by a Marie Curie International Research Staff Exchange Scheme Fellowship within the 7th European Community Framework Programme, FP7-PEOPLE-2012-IRSES-316338.



\end{document}